\documentclass[a4paper,12pt]{amsart}
\usepackage{amsmath,amssymb,amsfonts,amsthm,exscale,calc}
\usepackage{latexsym,textcomp}
\usepackage{a4wide}
\usepackage{dsfont}
\usepackage[german, english]{babel}
\usepackage{hyperref}

\newtheorem{lemma}{Lemma}[section]
\newtheorem{theorem}[lemma]{Theorem}

\newtheorem{proposition}[lemma]{Proposition}

\theoremstyle{definition}
\newtheorem{definition}[lemma]{Definition}
\newtheorem{example}[lemma]{Example}
\newtheorem*{remark}{Remark}

\numberwithin{equation}{section}
\newcommand{\comment}[1]{}

\newcommand{\cA}{{\mathcal A}}
\newcommand{\cB}{{\mathcal B}}

\newcommand{\cE}{{\mathcal E}}

\newcommand{\R}{{\mathbb R}}

\newcommand{\N}{{\mathbb N}}
\newcommand{\D}{{\mathbb D}}

\newcommand{\supp}{{\mathrm {supp}\,}}

\newcommand{\as}[1]{\left\langle #1\right\rangle}

\newcommand{\ow}[1]{\widetilde{ #1}}

\newcommand{\Hm}[1]{\leavevmode{\marginpar{\tiny%
$\hbox to 0mm{\hspace*{-0.5mm}$\leftarrow$\hss}%
\vcenter{\vrule depth 0.1mm height 0.1mm width \the\marginparwidth}%
\hbox to 0mm{\hss$\rightarrow$\hspace*{-0.5mm}}$\\\relax\raggedright
#1}}}

\newcommand{\1}{\mathds{1}}
\newcommand{\Eph}{\hat{\mathcal{E}}_\varphi}
\newcommand{\Ep}{\mathcal{E}_\varphi}
\newcommand{\Ee}{\mathcal{E}_{\rm e}}
\newcommand{\Em}{\mathcal{E}^{(M)}}
\newcommand{\Ek}{\mathcal{E}^{(k)}}
\newcommand{\Ekh}{\hat{\mathcal{E}}^{(k)}}
\newcommand{\Er}{\mathcal{E}^{\rm ref}}
\newcommand{\Era}{\mathcal{E}^{\rm ref}_a}
\newcommand{\tom}{\overset{m}{\to}} 
\newcommand{\toq}{\overset{\rho}{\to}} 

\begin{document}

\title{A note on reflected Dirichlet forms}

\author[Schmidt]{Marcel Schmidt}
\address{M. Schmidt, Mathematisches Institut \\Friedrich Schiller Universit{\"a}t Jena \\07743 Jena, Germany } \email{schmidt.marcel@uni-jena.de}

\begin{abstract}
In this paper we give an algebraic construction of the (active) reflected Dirichlet form. We prove that it is the maximal Silverstein extension whenever the given form does not possess a killing part and we prove that Dirichlet forms need not have a maximal Silverstein extension if a killing is present. For regular Dirichlet forms we provide an alternative construction of the reflected process on a compactification (minus one point) of the underlying space.
\end{abstract}


\maketitle



\section{Introduction}

Reflected forms of regular Dirichlet forms have been introduced by Silverstein in \cite{Sil1} to study the boundary behavior of Markov processes. More precisely, it is an  important question to characterize all extensions of a given Markov process beyond its lifetime, whose sample paths show the same local behavior as the original process. In terms of the associated Dirichlet forms Silverstein observed in \cite{Sil,Sil2} that such processes are encoded by Silverstein extensions  (with an additional condition on the generator) of the Dirichlet form  of the given process and that the active\footnote{The adjective active refers to to the fact that it lives on the space of square integrable functions. In the literature the reflected Dirichlet form is a form on all functions and therefore NOT a Dirichlet form; the active reflected Dirichlet form is then the restriction of the reflected Dirichlet form to square integrable functions.  Since this can be a source of confusion, we shall write of the active reflected Dirichlet form and of the reflected form (here we drop the word Dirichlet).} reflected Dirichlet form  is  maximal among all such extensions in the sense of quadratic forms. The name reflected form comes from the following observation. For Brownian motion on an Euclidean domain (with smooth boundary) that is killed upon leaving the domain the process associated to the active reflected Dirichlet form is Brownian motion reflected at the boundary of the domain, see e.g. \cite{Fuk67}.

 Silverstein proposed two approaches for constructing the reflected forms, one by extending processes and one by extending forms. However, he did not show that both approaches coincide and it seems that the closedness of the reflected forms in his 'form approach' is missing.  These gaps were closed by Chen in \cite{Che} with probabilistic methods. In \cite{Kuw} Kuwae extended the form theoretic approach to defining reflected forms to quasi-regular Dirichlet forms and obtained an analytic proof of the closedness of the reflected form. A streamlined version of both constructions for quasi-regular Dirichlet forms can now be found in the textbook \cite{CF} by Chen and Fukushima.

The known approaches to defining reflected forms have in common that they are rather technical and need some regularity of the underlying space and the Dirichlet form. One relies on the Beurling-Deny decomposition of quasi-regular Dirichlet forms while the other uses characterizations of harmonic functions through the associated Markov process. It is the main goal of this note to give  an 'algebraic' construction of reflected Dirichlet forms, which works for all Dirichlet forms. In spirit, a similar approach has been recently used by Robinson in \cite{Rob} for inner regular local Dirichlet forms but his methods cannot treat the nonlocal case and also needs some regularity. Note also that Robinson does not use the name reflected Dirichlet form; one of his extremal forms is the reflected form.  

Besides  greater generality our new approach has some virtues that lead to new insights even for regular Dirichlet forms. It is claimed in \cite{CF,Kuw} that the active reflected Dirichlet form is always the maximal Silverstein extension of the given quasi-regular Dirichlet form. Unfortunately, the statement and the given proofs are only correct if the form does not have a killing, see Proposition~\ref{proposition:counterexample} for a counterexample. We construct the reflected form by splitting the given form into its main and its killing part and extending both parts to the maximal possible domain. The reflected form is then the sum of both extensions. We obtain that the active reflected Dirichlet form is the maximal Silverstein extension if the killing vanishes (as discussed previously in this case the proofs in \cite{CF,Kuw} are also correct) but additionally prove that the active main part is the maximal form whose resolvent dominates the resolvent of the given form, see Theorem~\ref{theorem:maximality active reflected form}. This seems to be a new observation and leads to the insight that for every  Dirichlet form there exists a maximal Dirichlet form whose resolvent dominates the one of the given  form. Moreover, our construction allows us to prove that continuous functions are dense in the domain of the active reflected Dirichlet form of a regular Dirichlet form, see Theorem~\ref{theorem:continous functions are dense}. This can then be used to construct the reflected process on a compactification (minus one point) of the underlying space, see Theorem~\ref{theorem:regularity active main part}. To the best of our knowledge this precise topological information on the space where the reflected process can live is new. Previous constructions only show that there exists a locally compact space on which the reflected process lives that contains a quasi-open subset which is quasi-homeomorphic to the given underlying space. 

As can already be seen  for quasi-regular Dirichlet forms, reflected forms leave the realm of Dirichlet forms. They are Markovian forms on all a.e. defined measurable functions and are lower semicontinuous with respect to a.e. convergence. In our construction of reflected forms quadratic forms of this type feature even more prominently. While it is possible to show that they are extended Dirichlet forms after a change of the underlying measure, it is more natural to to replace a.e. convergence by local convergence in measure and consider them instead as closed Markovian forms on the topological vector space $L^0(m)$, so called energy forms. For localizable measures energy forms have been introduced by the author in his PhD thesis \cite{Schmi}. There it is shown that they are a common generalization of   extended Dirichlet forms and resistance forms in the sense of Kigami \cite{Kig2}.  In this note we only deal with $\sigma$-finite measures but with two types of quadratic forms: Dirichlet forms and energy forms.

The paper is organized as follows. In Section~\ref{section:preliminaries} we discuss basics about Dirichlet forms and energy forms. In particular, we clarify the relation of Silverstein extensions and form domination. Section~\ref{section:reflected forms} is devoted to the construction and properties of reflected forms. In Section~\ref{section:regular} we apply the developed theory to regular Dirichlet forms. In Appendix~\ref{appendix:closed forms on l0} we discuss the basics of closed forms on metrizable topological vectors spaces while Appendix~\ref{appendix:monotone forms} contains a characterization of monotone forms.

\medskip

Section~\ref{section:reflected forms} and Appendix~\ref{appendix:closed forms on l0} are based on the author's PhD thesis \cite{Schmi}. 

\medskip 

{\bf Acknowledgements:} The author would like to thank Alexander Grigor'yan for his encouragement to write this article. Large parts of the text were written while the author was enjoying the hospitality of Jun Masamune at Hokkaido University Sapporo. The support of JSPS for this stay is gratefully acknowledged. Since the present text is based on the author's PhD thesis, it also owes greatly to discussions with his former advisor Daniel Lenz.

\section{Preliminaries}\label{section:preliminaries}

 Throughout the paper $(X,\mathcal{B},m)$ is a $\sigma$-finite measure space. The space of real-valued measurable $m$-a.e. defined functions on $X$ is denoted by $L^0(m)$. We equip it with the vector space topology of {\em local convergence in measure}.  Recall that a sequence $(f_n)$ {\em converges to $f$ locally in measure} (in which case we write $f_n \overset{m}{\to} f$) if and only if for all sets $U \in \mathcal{B}$ with $m(U)  < \infty$ we have
$$\int_U |f -f_n| \wedge 1 dm \to 0 \text{, as } n\to \infty.$$
Here,  for $f,g \in L^0(m)$ we use the notation $f \wedge g$ for the pointwise minimum of  $f$ and $g$, and $f \vee g$ for the pointwise maximum.  Since $m$ is $\sigma$-finite, the topology of local convergence in measure is metrizable and $f_n \tom f$ if and only if  each subsequence of $(f_n)$ has a subsequence that converges $m$-a.e. to $f$, see e.g. \cite[Proposition~245K]{Fre2}. In particular, a.e. convergent sequences are locally convergent in measure.

In what follows we shall be concerned with Dirichlet forms on $L^2(m)$ and closed Markovian forms  on $L^0(m)$, so-called energy forms. Closed forms on topological vector spaces other than $L^2(m)$  seem to be not so well-studied. Therefore, we include a short discussion about them in Appendix~\ref{appendix:closed forms on l0}. We also refer to the beginning of this appendix for the general terminology that we use for quadratic forms. For a background on Dirichlet forms see e.g. \cite{CF,FOT,MR}.

\subsection{Dirichlet forms and domination}

Let $\cE$ be a {\em Dirichlet form}, i.e., a densely defined closed Markovian quadratic form on $L^2(m)$. We write $\as{\cdot,\cdot}_\cE$ for the {\em form inner product}
$$\as{f,g}_\cE = \cE(f,g) + \as{f,g}_2 \text{ for } f,g \in D(\cE),$$
where $\as{\cdot,\cdot}_2$ is the ordinary $L^2$-inner product on $L^2(m)$. The {\em form norm } is  $\|\cdot\|_\cE := \as{\cdot,\cdot}_\cE^{1/2}$. Recall the following structure properties of the domains of Dirichlet forms, see e.g. \cite[Theorem~I.4.12]{MR}.  

\begin{lemma} \label{lemma:contraction properties}
 Let $\cE$ be a Dirichlet form. For $f,f_1,\ldots,f_n \in L^2(m)$ the  inequalities
$$|f(x)| \leq \sum_{k=1}^n|f_k(x)| \text{ and } |f(x)-f(y)| \leq \sum_{k = 1}^n  |f_k(x)-f_k(y)| \text{ for }  m\text{-a.e. } x,y \in X $$
 imply
$$\cE(f)^{1/2} \leq \sum_{k = 1}^n \cE(f_k)^{1/2}.$$
In particular, $D(\cE) \cap L^\infty(m)$ is an algebra and $D(\cE)$ is a lattice.
\end{lemma}

Let $\cE$ and $\tilde{\cE}$ be Dirichlet forms on $L^2(m)$ and let $(G_\alpha)_{\alpha > 0}$, respectively $(\tilde G_\alpha)_{\alpha> 0}$, be the associated Markovian resolvents. We say that {\em $\tilde{\cE}$ dominates $\cE$} if $(\tilde G_\alpha)_{\alpha> 0}$ dominates $(G_\alpha)_{\alpha > 0}$, i.e., if for all $f \in L^2(m)$ and all $\alpha>0$ the inequality $|G_\alpha f| \leq \tilde{G}_\alpha |f|$ holds.

Domination of forms can be characterized by ideal properties of their domains. This is discussed next. For two subsets $I,S \subseteq L^0(m)$ we say that $I$ is an {\em order ideal in $S$} if $f \in S$, $g \in I$ and $|f| \leq |g|$ implies $f \in I$ and we say that $I$ is an {\em algebraic ideal in $S$} if $f \in S$, $g \in I$ implies $fg \in I$. 
\begin{lemma} \label{lemma:characterization of domination}
 Let $\cE$, $\tilde{\cE}$  Dirichlet forms on $L^2(m)$. The following assertions are equivalent.
 \begin{itemize}
  \item[(i)] $\cE$ is dominated by $\tilde \cE$.
  
  \item[(ii)]  $D(\cE) \subseteq D(\tilde \cE)$, $D(\cE)$ is an order ideal in $D(\tilde{\cE})$  and  
  $$\cE(f,g) \geq \tilde{\cE}(f,g)$$
  for all nonnegative $f,g \in D(\cE)$. 
  \item[(iii)] $D(\cE) \subseteq D(\tilde \cE)$, $D(\cE)\cap L^\infty(m)$ is an algebraic ideal in $D(\tilde{\cE}) \cap L^\infty(m)$ and  
  $$\cE(f,g) \geq \tilde{\cE}(f,g)$$
  for all nonnegative $f,g \in D(\cE)$. 
 \end{itemize}
\end{lemma}
\begin{proof}
 (i) $\Leftrightarrow$ (ii): This follows from \cite[Corollary~4.3]{MVV}, where instead of the resolvents the associated semigroups are considered. It is readily verified that domination of the associated resolvents is equivalent to domination of the associated semigroups.
 
 (ii) $\Rightarrow$ (iii): Let $f \in D(\tilde{\cE}) \cap L^\infty(m)$ and let $g \in D(\cE)\cap L^\infty(m)$. We obtain $|fg| \leq \|f\|_\infty |g|$. Since $D(\cE)$ is an order ideal in $D(\tilde \cE)$, this implies $fg \in D(\cE)$.
 
 (iii) $\Rightarrow$ (ii): Let $f \in D(\tilde \cE)$ and let $g \in D(\cE)$ with $|f| \leq |g|$. Since domains of Dirichlet forms are lattices, we can assume $0 \leq f \leq g$. Moreover, since $f \wedge n \to  f$, $g \wedge n \to g$, as $n \to \infty$, with respect to the corresponding form norms, see e.g. \cite[Theorem~1.4.2]{FOT}, we can further assume that $f$ and $g$ are bounded.
 
 Let $A:= \{(x,y) \in \R^2 \mid 0 \leq x \leq y\}$,   $\varepsilon > 0$ and consider
 $$C_\varepsilon:A \to \R,\, C_\varepsilon(x,y) := x \frac{y}{y+\varepsilon}.$$
 For $(x_i,y_i) \in A$, $i=1,2$, it satisfies
 $$|C_\varepsilon(x_1,y_1) - C_\varepsilon(x_2,y_2)| \leq |x_1 - x_2| + |y_1 - y_2|$$
 and $C_\varepsilon(0,0) = 0$. Since $0\leq f \leq g$, we also have 
 $$C_\varepsilon(f,g) = f \frac{g}{g + \varepsilon} \to f \text{ in } L^2(m),\text{ as }\varepsilon \to 0+,$$
 and the $L^2$-lower semicontinuity of $\cE$ implies
 $$\cE(f) \leq \liminf_{\varepsilon\to 0+} \cE(C_\varepsilon(f,g)).$$
 Therefore, it suffices to prove that the right-hand side of this inequality is bounded independently of $\varepsilon$. 
 
 The function $H_\varepsilon:[0,\infty) \to \R, H_\varepsilon(y) = y(y+\varepsilon)^{-1}$ is $\varepsilon^{-1}$-Lipschitz and satisfies $H_\varepsilon(0) = 0$, such that  $H_\varepsilon(g) \in D(\cE)$ by  Lemma~\ref{lemma:contraction properties}.  Since $D(\cE)\cap L^\infty(m)$ is an algebraic ideal in $D(\tilde{\cE}) \cap L^\infty(m)$, this implies $C_\varepsilon(f,g)  = f H_\varepsilon(g)\in D(\cE)$ for all $\varepsilon > 0$. The inequality between $\cE$ and $\tilde{\cE}$ for nonnegative functions in $D(\cE)$ and the inequalities $0 \leq C_{\varepsilon}(f,g) \leq f \leq g$  then show
 \begin{align*}
  \cE(C_\varepsilon(f,g)) &= \tilde \cE(C_\varepsilon(f,g)) + \cE(C_\varepsilon(f,g))   - \tilde \cE(C_\varepsilon(f,g)) \\
  &\leq  \tilde \cE(C_\varepsilon(f,g)) + \cE(g) - \tilde \cE(g).
 \end{align*}
 Moreover, the properties of $C_\varepsilon$ and Lemma~\ref{lemma:contraction properties} imply
 $$\tilde  \cE(C_\varepsilon(f,g))^{1/2} \leq \tilde \cE (f)^{1/2} + \tilde \cE (g)^{1/2}.$$
 Altogether we obtain that $\cE(C_\varepsilon(f,g))$ is bounded independently of $\varepsilon$. This finishes the proof.
\end{proof}
\begin{remark}
 For the case when $\tilde \cE$ is an extension of $\cE$ the author learned the presented proof of (iii) $\Rightarrow$ (ii) in a discussion with Peter Stollmann and Hendrik Vogt in December 2012. Independently, a different proof for the  lemma was  recently given in \cite{Rob} in the somewhat more restrictive setting that $X$ is a locally compact separable metric space and $m$ is a Radon measure of full support on $X$. 
\end{remark}
If the quadratic form $\tilde \cE$ is an extension of $\cE$, then $D(\cE) \subseteq D(\tilde \cE)$ and the inequality $\cE(f,g) \geq \tilde \cE(f,g)$ for nonnegative $f,g \in D(\cE)$ are trivially satisfied. If, in this case, the form $\tilde \cE$ satisfies any of the three equivalent conditions of the previous lemma, it is called {\em Silverstein extension} of $\cE$. It is one main goal of this paper to construct for a given Dirichlet form the maximal Dirichlet form that dominates it and, if possible, to also construct its maximal Silverstein extension.  
\begin{remark}
 Silverstein extensions play an important rôle in the study of the boundary behavior of symmetric Markov processes. Their study was initiated in the books \cite{Sil,Sil2}. A modern treatment can be found in \cite{CF}.
\end{remark}
\subsection{Energy forms and extended Dirichlet spaces} 

As discussed in Appendix~\ref{appendix:closed forms on l0}, a quadratic form $E$ on $L^0(m)$ is called {\em closed} if it is lower semicontinuous with respect to local convergence in measure, i.e., if for all sequences $(f_n)$ in $L^0(m)$ and all $f \in L^0(m)$ the convergence $f_n \tom f$ implies
$$E(f) \leq \liminf_{n\to \infty} E(f_n).$$

A closed quadratic form $E$ on $L^0(m)$ is called {\em energy form} if it is {\em Markovian}, i.e., if for each normal contraction $C:\R \to \R$ and all $f \in L^0(m)$ we have
$$E(C \circ f) \leq E(f).$$
Clearly, the restriction of an energy form to $L^2(m)$ is a (not necessarily densely defined) Dirichlet form on $L^2(m)$ as its $L^2$-lower semicontinuity follows from the continuity of the embedding $L^2(m) \hookrightarrow L^0(m)$. The theory of extended Dirichlet forms shows that also the 
opposite way is possible.

A Dirichlet form $\cE$ on $L^2(m)$ can be   considered to be a quadratic form on $L^0(m)$ by letting $\cE(f) = \infty$ for $f \in L^0(m) \setminus L^2(m)$. The next lemma shows that this form is closable and its closure is an energy form, the so-called {\em extended Dirichlet form}, which we denote by $\Ee$.  
\begin{lemma}
 Every Dirichlet form on $L^2(m)$ is closable on $L^0(m)$. Its closure $\Ee$ is an energy form that is given by
 $$\Ee:L^0(m) \to [0,\infty],\, \Ee(f):= \begin{cases}
                 \lim\limits_{n \to \infty} \cE(f_n) &\text{if } (f_n) \text{ is }\cE\text{-Cauchy with } f_n \tom f,\\
                 \infty &\text{if there exists no such sequence}.
                \end{cases} $$
\end{lemma}
\begin{proof}
For proving closability and the formula for $\Ee$ we employ Lemma~\ref{lemma:characterization closability}. Thus, we need to show that $\cE$ is lower semicontinuous with respect to local convergence in measure on its domain. To this end, let $(f_n)$ a sequence in $D(\cE)$ and let $f \in D(\cE)$ with $f_n \tom f$. Without loss of generality we can assume $\liminf_{n \to \infty} \cE(f_n) < \infty$ and by passing to a suitable subsequence we can further assume $f_n \to f$ $m$-a.e. It then follows from the main result of \cite{Schmu2} that  $\cE(f) \leq \liminf_{n \to \infty} \cE(f_n).$ For a simplified proof see \cite[Theorem~1.59]{Schmi}. 

It remains to show that $\Ee$ is Markovian. Let $C:\R \to \R$ a normal contraction and let $f \in D(\Ee)$. We choose an $\cE$-Cauchy sequence $(f_n)$ in $D(\cE)$ with $f_n \tom f$. The lower semicontinuity of $\Ee$,  the Markov property of $\cE$ and the fact that $\Ee$ extends $\cE$ yield
$$\Ee(C \circ f) \leq \liminf_{n\to \infty} \Ee(C \circ f_n)  =   \liminf_{n\to \infty} \cE(C \circ f_n) \leq  \liminf_{n\to \infty} \cE( f_n) = \Ee(f).$$
This finishes the proof.
\end{proof}
It follows  from the previous lemma and the characterization of local convergence in measure that the domain of $\Ee$ coincides with the classical extended Dirichlet space, i.e.,
$$D(\Ee) = \{f \in L^0(m) \mid \text{ there ex.  $\cE$-Cauchy sequence } (f_n) \text{ in } D(\cE) \text{ with } f_n \to f\, m\text{-a.e.}\}.$$
In particular, $D(\Ee) \cap L^2(m)$ = $D(\cE)$, see e.g. \cite[Theorem~1.1.5]{CF}.
 
 \begin{remark}
  We included the previous lemma because it seems that the lower semicontinuity of $\Ee$ on $L^0(m)$ is not contained in the literature. We could only find the inequality 
  $$\Ee(f) \leq \liminf_{n\to \infty} \Ee(f_n) $$
   for sequences $(f_n)$ with $f_n \to  f$ $m$-a.e. (which is almost the same as local convergence in measure) under the additional assumption that either $f \in D(\Ee)$, see \cite[Corollary~1.9]{CF}, or $f_n \in D(\cE)$, see \cite[Lemma~2]{Schmu2}. In contrast, the Markov property of $\Ee$ is contained in the literature, see e.g. \cite[Theorem~1.1.5]{CF}. It seems that our proof, which makes direct use of lower semicontinuity, is a bit simpler.  
 \end{remark}
 
 \begin{remark}
Energy forms have been introduced in \cite{Schmi} on localizable measure spaces.  It can  be proven that for a given $\sigma$-finite measures $m$,  any energy form on $L^0(m)$ is an extended Dirichlet form of a not necessarily densely defined Dirichlet form (a so-called Dirichlet form in the wide sense). However, the measure $m$ needs to be changed to a suitable equivalent finite measure $m'$ on $X$ for which $L^0(m) = L^0(m')$ as topological vector spaces (e.g. one can take  $m' = g \cdot m$ for some strictly positive $g \in L^1(m)$). For details see \cite[Proposition~3.7]{Schmi}. Below we will  construct different energy forms on $L^0(m)$. Since it would be somewhat artificial and technically more complicated to consider them as extended Dirichlet forms with respect to a changed measure, we will directly work in the category of energy forms. Note also that energy forms on non-$\sigma$-finite measure spaces need not be extended Dirichlet forms. For example, resistance forms in the sense of Kigami \cite{Kig2} are energy forms on a set equipped with the counting measure, see the discussion in \cite[Subsection~2.1.3]{Schmi}. 
 \end{remark}
It follows from the previous remark that the statements of Lemma~\ref{lemma:contraction properties} are also true for energy forms, see also \cite[Theorem~2.20]{Schmi} for a direct proof. We only need the following consequences which can be proven directly.
\begin{lemma}\label{lemma:bounded approximation}
 Let $E$ be an energy form on $L^0(m)$. Let $f \in D(E)$ and for $\alpha > 0$ let $f^{(\alpha)} := (f\wedge \alpha)\vee(-\alpha)$. Then $f^{(\alpha)} \in D(E)$ and the following convergence statements hold true.
 \begin{itemize}
  \item[(a)]  $E(f^{(n)} - f)  \to $, as $n\to \infty$. 
  \item[(b)] Let $(f_n)$ a sequence in $D(E)$ with $f_n \tom f$ and $E(f_n - f)\to 0$, as $n \to \infty$. Then, for every $\alpha \geq \|f\|_\infty$, we have $f^{(\alpha)}_n \tom f$ and $E(f^{(\alpha)}_n - f) \to 0$, as $n \to \infty$.
  
 \end{itemize}
\end{lemma}
\begin{proof}
 (a) + (b): All statements are consequences of the Markov property of $E$ (applied to the normal contraction $\R \to \R, x\mapsto (x \wedge \alpha) \vee(-\alpha)$), its lower semicontinuity and Lemma~\ref{lemma:existence of a weakly convergent subnet}.
\end{proof}

\begin{lemma}\label{lemma:algebraic properties}
 Let $E$ be an energy form. Then $D(E)$ a lattice and for every $f,g \in D(E)$ we have
 $$E(f\wedge g)^{1/2} \leq E(f)^{1/2} + E(g)^{1/2} \text{ and } E(f\vee g)^{1/2} \leq E(f)^{1/2} + E(g)^{1/2}.$$
 Moreover, $D(E) \cap L^\infty(m)$ is an algebra and for every $f,g \in D(E)\cap L^\infty(m)$ the inequality
  $$E(fg)^{1/2} \leq \|f\|_\infty E(g) + \|g\|_\infty E^{1/2}(f)$$
  holds.
\end{lemma}
\begin{proof}
  The first statement follows from the identity $f \wedge g = \frac{f + g - |f-g|}{2}$ and the fact that $\R \to \R,x \mapsto |x|$ is a normal contraction. $f \vee g$ can be treated similarly. 

 For the 'Moreover'-statement let $h \in L^1(m)$ with $h> 0$ $m$-a.e. and set $m' = h \cdot m$. Then the embedding $L^2(m') \hookrightarrow L^0(m)$ is continuous and so the restriction of $E$ to $L^2(m')$, which we denote by $\cE$, is a (not necessarily densely defined) Dirichlet form. Since $h \in L^1(m)$ and $f,g$ are bounded, we have $f,g \in L^2(m')$ and therefore $f,g \in D(\cE)$. Now the statement follows from \cite[Theorem~1.4.2]{FOT}. 
\end{proof}

 We call an energy form $E$ {\em recurrent} if $1 \in \ker E$ and {\em transient if $\ker E = \{0\}$}. These notions are borrowed from Dirichlet form theory, as a Dirichlet form $\cE$ is recurrent  if and only if $1 \in \ker \Ee$ and it is transient if and only if $\ker \Ee = \{0\}$, see e.g. \cite[Theorem~1.6.2 and Theorem~1.6.3]{FOT}.

\section{Reflected Dirichlet forms} \label{section:reflected forms}

In this section we construct the reflected Dirichlet form by splitting the given Dirichlet form into its main part and its killing part and then extending both parts to the maximal possible domain. We prove that the $L^2$-restriction of the main part is the maximal Dirichlet form that dominates the given form and we prove that it is the maximal Silverstein extension if there is no killing part. Moreover, we give an example of a Dirichlet form with non-vanishing killing part that does not possess a maximal Silverstein extension. Everything in this section is based on \cite[Chapter~3.3]{Schmi}.

\subsection{The main part}

In this subsection $\cE$ is a fixed Dirichlet form. Let $\varphi \in D(\cE)$ with $0 \leq \varphi \leq 1$ be given. We define the functional $\Eph:L^0(m) \to [0,\infty]$ by
$$\Eph(f) := \begin{cases}
              \cE(\varphi f) - \cE(\varphi f^2,\varphi) & \text{if } f \in L^\infty(m), \varphi f, \varphi f^2 \in D(\cE),\\
              \infty &\text{else}.
             \end{cases}
$$
On a technical level the following theorem is the main insight of this subsection. We formulate it as a theorem because we think that it has some applications  beyond the construction of reflected Dirichlet forms. Recall that the relation $\leq$ on quadratic forms compares the size of their domains; we have $q \leq q'$ for two quadratic forms $q,q'$ if and only if $q(f) \geq q'(f)$ for all $f \in D(q)$, cf. Appendix~\ref{appendix:closed forms on l0}.
\begin{theorem}[Properties of truncated forms] \label{theorem:properties of concatenated forms}
The functional $\Eph$ is a closable quadratic form on $L^0(m)$. Its closure $\Ep$ is a recurrent energy form with the following properties.
\begin{itemize}
 \item[(a)] $\Ee \leq \Ep$ and if $\psi \in D(\cE)$ with $\varphi \leq \psi \leq 1$, then $\cE_\psi \leq \Ep$.
 \item[(b)] $D(\Ep) \cap L^\infty(m) = D(\Eph) = \{f \in L^\infty(m) \mid \varphi f \in D(\cE)\}$. 
 \item[(c)] If $\tilde \cE$ is a Dirichlet form that dominates $\cE$, then $\tilde \cE_\varphi \leq \Ep$.
\end{itemize}
\end{theorem}
 \begin{definition}[Truncated form]
  The closure $\Ep$ of $\Eph$ on $L^0(m)$ is called the {\em truncation of $\cE$ with respect to $\varphi$}. 
 \end{definition}
 \begin{remark}
  Property (b) in the theorem is quite important as it shows how to compute   $\Ep$ and its domain. Namely, if $f \in L^\infty(m)$ with $\varphi f \in D(\cE)$ we have
  $$\Ep(f) = \Eph(f)  = \cE(\varphi f) - \cE(\varphi f^2,\varphi).$$
  For arbitrary $f \in L^0(m)$ and $n \in \N$ we let $f^{(n)}:= (f\wedge n) \vee(-n).$ It follows from the Markov property and the lower semicontinuity of $\Ep$ that $\Ep(f) = \lim\limits_{n\to \infty}\Ep(f^{(n)})$. Therefore, $f \in D(\Ep)$ if and only if $\varphi f^{(n)} \in D(\cE)$ for each $n \in \N$ and the limit
  $$\Ep(f) = \lim_{n\to \infty} \cE(\varphi f^{(n)}) - \cE(\varphi (f^{(n)})^2,\varphi) $$
  is finite.
 \end{remark}

 In order to prove this theorem we need two lemmas.
\begin{lemma} \label{lemma:maximal silverstein extension technical lemma}
 Let $\varphi \in D(\cE)$ with $0 \leq \varphi \leq 1$ and let $f \in L^\infty(m)$. 
 \begin{itemize}
  \item[(a)] Let $C:\R \to \R$ be $L$-Lipschitz with $C(0) = 0$ and let $$M:= \sup \{|C(x)| \mid |x|\leq \|f\|_\infty\}.$$ Then
  $$\cE( \varphi\, C(f))^{1/2} \leq L \cE(\varphi f)^{1/2} + (M + L\|f\|_\infty) \cE(\varphi)^{1/2}.$$

  In particular, $\varphi f  \in D(\cE)$ implies $\varphi\, C(f)  \in D(\cE)$ and $\varphi f^2 \in D(\cE)$. 
  \item[(b)] If $\psi \in D(\cE)$ with $0 \leq \varphi \leq \psi$, then $\psi f  \in D(\cE)$ implies $\varphi f  \in D(\cE)$. 
 \end{itemize}
\end{lemma}
\begin{proof}
 (a): Let $A :=\{(x,y) \in \R^2 \mid |x| \leq |y| \cdot \|f\|_\infty\}$ and consider the function 
$$\ow{C}: A \to \R, \quad (x,y) \mapsto \ow{C}(x,y) := \begin{cases}C\left({x}/{y}\right)y &\text{if } y \neq 0\\0 & \text{if }  y = 0 \end{cases}.$$
We show that $\ow{C}$ is Lipschitz with appropriate constants. The statement  then follows from Lemma~\ref{lemma:contraction properties} and the identity $\varphi\, C(f) = \ow{C}(\varphi f, \varphi)$. For $(x_1,y_1),(x_2,y_2) \in A$ with $y_1,y_2 \neq 0$ we have
\begin{align*}
 |\ow{C}(x_1,y_1) - \ow{C}(x_2,y_2)| &\leq |y_1| |C(x_1/y_1) - C(x_2/y_2)| + |C(x_2/y_2)| |y_1 - y_2|.
\end{align*}
Since $|x_i|\leq |y_i| \|f\|_\infty$, $i=1,2$, we obtain $|C(x_2/y_2)| \leq M$ and 
$$|C(x_1/y_1) - C(x_2/y_2)| \leq L \left|\frac{x_1y_2 - x_2 y_1}{y_1y_2}\right| \leq L \left|\frac{x_1  - x_2}{y_1}\right| + L \|f\|_\infty \left|\frac{y_1 - y_2}{y_1}\right|. $$
Altogether, these considerations amount to 
$$|\ow{C}(x_1,y_1) - \ow{C}(x_2,y_2)| \leq L|x_1 - x_2| + (M + L \|f\|_\infty) |y_1 - y_2|.$$
By the continuity of $\ow{C}$ on $A$, this inequality extends to the case when $y_2 = x_2 = 0$, in which it reads  
$$|\ow{C}(x_1,y_1)| \leq L|x_1| + (M + L \|f\|_\infty) |y_1|.$$
From Lemma~\ref{lemma:contraction properties} we infer
$$\cE(\varphi\, C(f))^{1/2} = \cE(\ow{C}(\varphi f,\varphi))^{1/2} \leq L \cE(\varphi f)^{1/2} + (M + L\|f\|_\infty) \cE(\varphi)^{1/2}. $$
For the 'In particular'-part we apply the statement to the function 
$$C:\R \to \R,\, x \mapsto x^2 \wedge \|f\|_\infty^2.$$

(b):  We let $B:=\{(x,y,z) \in  \R^3  \mid z \geq 0, |x| \leq z \cdot \|f\|_\infty \text{ and } |y| \leq z\}$. For $\varepsilon> 0$  we consider the function
$$C_\varepsilon:B \to \R,\quad (x,y,z) \mapsto C_\varepsilon(x,y,z):= xy/(z + \varepsilon).$$
From the inequality $0 \leq \varphi \leq \psi$ we obtain
$$C_\varepsilon(\psi f, \varphi ,\psi) =  \varphi  \frac{\psi}{\psi + \varepsilon} f \to  \varphi f $$
in $L^2(m)$,   as $\varepsilon \to 0+$. The lower semicontinuity of $\cE$ implies
$$\cE(\varphi f) \leq \liminf_{\varepsilon \to 0+} \cE(C_\varepsilon(\psi f ,\varphi,\psi)). $$
Thus, it suffices to prove that the right-hand side of the above inequality is finite. 

The partial derivatives of $C_\varepsilon$ satisfy $|\partial_x C_\varepsilon| \leq 1$, $|\partial_y C_\varepsilon| \leq \|f\|_\infty$ and $|\partial_z C_\varepsilon| \leq \|f\|_\infty$ in the interior of $B$. This yields
$$|C_\varepsilon(x_1,y_1,z_1) - C_\varepsilon(x_2,y_2,z_2)| \leq |x_1 - x_2| + \|f\|_\infty |y_1-y_2| + \|f\|_\infty |z_1 - z_2|,$$
 for $(x_i,y_i,z_i), i = 1,2,$ in the interior of $B$. Since  $C_\varepsilon$ is continuous and any point in $B$ can be approximated by interior points, we can argue similarly as in the proof of assertion (a) to obtain 
$$\cE(C_\varepsilon(\psi f, \varphi,\psi))^{1/2} \leq \cE(\psi f)^{1/2} + \|f\|_\infty \cE(\varphi)^{1/2} + \|f\|_\infty \cE(\psi)^{1/2} < \infty. $$
This finishes the proof.
\end{proof}

\begin{lemma}\label{lemma:properties of eph} Let $\varphi,\psi \in D(\cE)$ with $0\leq \varphi,\psi \leq 1$. 
 \begin{itemize}
 \item[(a)] $\Eph$ is a nonnegative quadratic form on $L^0(m)$. Its domain satisfies
 $$D(\Eph) = \{f \in L^\infty(m) \mid  \varphi f \in D(\cE)\}.$$
 \item[(b)] For every normal contraction $C:\R \to \R$ and every $f \in L^0(m)$ the inequality 
 $$\Eph(C\circ f) \leq \Eph(f)$$
  holds. 
  \item[(c)] If $ \varphi \leq \psi$, then for all $f \in L^0(m)$ we have $\Eph (f) \leq \hat{\cE}_{\psi}(f)$. 
  \item[(d)] For all $f \in D(\cE) \cap L^\infty(m)$ the inequality $\Eph (f) \leq \cE(f)$ holds.
  \item[(e)] If $\tilde \cE$ is a Dirichlet form that dominates $\cE$, then for all $f \in L^0(m)$ we have $$\Eph(f) \leq \hat{\tilde \cE}_\varphi(f).$$
 \end{itemize}
 
 \end{lemma}

\begin{proof}
(a): Let $f \in L^\infty(m)$. Lemma~\ref{lemma:maximal silverstein extension technical lemma}~(a) shows that $\varphi  f\in D(\cE)$ implies $\varphi f^2  \in D(\cE)$. This observation and the fact that $\cE$ is a quadratic form then show that $\Eph$ is a quadratic form whose domain satisfies $D(\Eph) = \{f \in L^\infty(m) \mid \varphi f \in D(\cE)\}$. The nonnegativity of $\Eph$ follows from assertion (c) (we shall not use this fact in the rest of the proof).

Let $(G_\alpha)$ be the resolvent of $\cE$. We denote the corresponding continuous approximating form by $\cE^{(\alpha)}$, i.e.,   
$$\cE^{(\alpha)}:L^2(m) \to [0,\infty),\, \cE^{(\alpha)}(f) :=  \as{f, (I-\alpha G_\alpha)f}.$$
It is shown in \cite{FOT} that for $f \in L^2(m)$ we have 
$$\cE(f) = \lim_{\alpha \to \infty} \alpha \cE^{(\alpha)}(f),$$
where the limit is infinite on $L^2(m) \setminus D(\cE)$. In particular, for $f  \in D(\Eph)$ this implies
$$\Eph(f) = \cE(\varphi f) - \cE(\varphi f^2,\varphi) = \lim_{\alpha\to \infty} \alpha \left(\cE^{(\alpha)}(\varphi f) - \cE^{(\alpha)}(\varphi f^2,\varphi)\right)=  \lim_{\alpha\to \infty} \alpha \hat \cE_\varphi^{(\alpha)}(f). $$
Note that since $\Eph(f)$ involves taking off-diagonal values of $\cE$, this approximation for $\Eph$ is only valid for functions in the domain   $D(\Eph)$.

We now prove that assertions (b), (c), (d)  and (e) hold true for the continuous Dirichlet form $\cE^{(\alpha)}$ and then infer the statement for general forms by an approximation procedure. To simplify notation we write $\hat \cE^{(\alpha)}_\varphi$ for the form $\hat{(\cE^{(\alpha)})}_\varphi$. 

Since $\varphi \in D(\cE) \subseteq L^2(m),$ for any $f \in L^\infty(m)$ we have $\varphi f \in L^2(m) = D(\cE^{(\alpha)})$ and so $D(\hat \cE^{(\alpha)}_\varphi) = L^\infty(m)$. Any function in $L^\infty(m)$ can be approximated by a sequence of simple functions $(f_n)$ in $L^2(m)$ such that $\varphi f_n$ converges to $\varphi f$  and $\varphi f_n^2$ converges to  $\varphi f_n^2$ in $L^2(m)$. Therefore, it suffices to prove the statements for simple $L^2$-functions. To this end, let
$$f = \sum_{j = 1}^n \alpha_j \1_{A_j}$$
with pairwise disjoint $A_j$ of finite measure be given. We obtain 
$$\hat \cE^{(\alpha)}_\varphi(f) = \sum_{i,j=1}^n b_{ij}^\varphi (\alpha_i - \alpha_j)^2 +  \sum_{i}^n c_{i}^\varphi \alpha_i^2,$$
with 
$$b_{ij}^\varphi = -\cE^{(\alpha)}_\varphi(\1_{A_i},\1_{A_j}) = -\cE^{(\alpha)}(\varphi \1_{A_i},\varphi\1_{A_j}) = \alpha \as{\varphi \1_{A_i},G_\alpha (\varphi \1_{A_j})} $$
and 
$$c_{i}^\varphi = \cE^{(\alpha)}_\varphi(\1_{A_i},\1_{\cup_j A_j}) = \cE^{(\alpha)}(\varphi \1_{A_i},\varphi (\1_{\cup_j A_j} - 1)) = \as{\varphi \1_{A_i},  \alpha  G_\alpha (\varphi  \1_{X \setminus \cup_j A_j})}.$$
The same computation for $\cE^{(\alpha)}$ yields 
$$\cE^{(\alpha)}(f) = \sum_{i,j=1}^n b_{ij}  (\alpha_i - \alpha_j)^2 +  \sum_{i}^n c_{i} \alpha_i^2,$$
with 
$$b_{ij} = \cE^{(\alpha)}(\1_{A_i},\1_{A_j}) = \alpha \as{ \1_{A_i},G_\alpha  \1_{A_j}} $$
 and 
 $$c_i = \cE^{(\alpha)} (\1_{A_i},\1_{\cup_j A_j}) = \as{\1_{A_i}, \1_{\cup_j A_j} - \alpha G_\alpha\1_{\cup_j A_j}}.  $$
 Since $\alpha G_\alpha$ is Markovian, these identities show 
 $0 \leq b_{ij}^\varphi \leq  b_{ij} \text{ and } 0 \leq c_i^\varphi \leq c_i,$
 and we obtain (b) and (d) for the form $\cE^{(\alpha)}$ (cf. the proof of \cite[Theorem~I.4.12]{MR}). If $\psi \in L^2(m)$ with $\varphi \leq \psi \leq 1$, then we also have $b_{ij}^\varphi \leq b_{ij}^\psi$ and $c_i^\varphi \leq c_i^\psi$ proving (c) for the form $\cE^{(\alpha)}$. If $\tilde \cE$ is a Dirichlet form that dominates $\cE$, it follows from the formula for $b_{ij}^\varphi$ and $c_i^\varphi$ that $ \hat \cE^{(\alpha)}_\varphi(f)  \leq \hat{\tilde \cE}^{(\alpha)}_\varphi(f) $ for each $f \in L^2(m)$.

 We now prove the statements (b), (c), (d) and (e) for  $\Eph$ by approximating it with $\hat \cE^{(\alpha)}_\varphi$. Since this approximation is only valid on $D(\Eph)$, for each statement we still need to verify that the involved functions belong to the correct domain.   
 
 (b): Let $C:\R \to \R$ be a normal contraction and let $f \in D(\Eph)$ such that $\varphi f,\varphi f^2 \in D(\cE)$. Lemma~\ref{lemma:maximal silverstein extension technical lemma} yields $\varphi\, C(f) \in D(\cE)$ and $\varphi\, C(f)^2 \in D(\cE)$. Thus, using (b) for the approximating forms, we obtain
 $$\Eph(C(f)) = \lim_{\alpha \to \infty} \alpha \hat \cE^{(\alpha)}_\varphi(C(f)) \leq  \lim_{\alpha \to \infty} \alpha \hat \cE^{(\alpha)}_\varphi( f ) = \Eph(f).$$

 (c): Let $f \in D(\hat \cE_\psi)$. Since $\psi f,\psi f^2 \in D(\cE)$, Lemma~\ref{lemma:maximal silverstein extension technical lemma} yields $\varphi f,\varphi  f^2 \in D(\cE)$. Thus, using (c) for the approximating forms  yields
 $$\Eph( f ) = \lim_{\alpha \to \infty} \alpha \hat \cE^{(\alpha)}_\varphi(f) \leq  \lim_{\alpha \to \infty} \alpha \hat \cE^{(\alpha)}_\psi( f ) = \cE_\psi( f ). $$
 (d): Let $f \in D(\cE) \cap L^\infty(m)$. Since $D(\cE) \cap L^\infty(m)$ is an algebra, we have $\varphi f \in D(\cE)$ and Lemma~\ref{lemma:maximal silverstein extension technical lemma} yields $\varphi f^2 \in D(\cE)$. Using (d) for the approximating forms shows
 $$\Eph(f) = \lim_{\alpha \to \infty} \alpha \hat \cE^{(\alpha)}_\varphi(f) \leq  \lim_{\alpha \to \infty} \alpha \hat \cE^{(\alpha)} ( f ) = \cE( f ).$$

 (e): Let $f \in D(\hat{\tilde \cE}_\varphi)$ such that $\varphi f \in D(\tilde \cE)$. We have $\varphi \in D(\cE)$ and $|\varphi f| \leq \|f\|_\infty \varphi$. Since $D(\cE)$ is an order ideal in $D(\tilde \cE)$, this implies $\varphi f\in D(\cE)$ and also $\varphi f^2 \in D(\cE)$ by Lemma~\ref{lemma:maximal silverstein extension technical lemma}. Using (e) for the approximating forms yields
 $$\Eph(f) = \lim_{\alpha \to \infty} \alpha \hat \cE^{(\alpha)}_\varphi(f) \leq \lim_{\alpha \to \infty} \alpha  \hat{\tilde \cE}^{(\alpha)}_\varphi(f) = \hat{\tilde \cE}_\varphi(f).$$
 This finishes the proof.
 \end{proof}
 
   We can now prove Theorem~\ref{theorem:properties of concatenated forms} by showing that $\Eph$ is closable and that the properties discussed in the previous lemma pass to its closure.

 \begin{proof}[Proof of Theorem~\ref{theorem:properties of concatenated forms}]
  We first prove that $\Eph$ is closable on $L^0(m)$. Indeed, we show that for any sequence $(f_n)$ in $L^\infty(m)$ and $f \in L^\infty(m)$ the convergence $f_n \tom f$ implies
  $$\Eph(f) \leq \liminf_{n\to \infty} \Eph(f_n).$$
  This means that the restriction of $\Eph$ to $L^\infty(m)$ is lower semicontinuous with respect to $L^0(m)$-convergence. It is slightly stronger than closability (cf. Lemma~\ref{lemma:characterization closability}) and is crucial for proving the identity $D(\Ep) \cap L^\infty(m) = D(\Eph)$ later on.
  
  If $\liminf_{n\to \infty} \Eph(f_n) = \infty$ there is nothing to show. Hence, after passing to a suitable subsequence, we can assume  that $(f_n)$ in $D(\Eph)$ and
  $$\lim_{n\to \infty}\Eph(f_n) = \liminf_{n\to \infty} \Eph(f_n) < \infty.$$
  Moreover, by Lemma~\ref{lemma:properties of eph}~(b) we can further assume $\|f_n\|_\infty \leq \|f\|_\infty$.
  
  Lemma~\ref{lemma:maximal silverstein extension technical lemma}~(a) applied to the function $C:\R \to \R, \, C(x) = x^2 \wedge \|f\|_\infty ^2$ yields
 $$\cE(\varphi f_n^2)^{1/2} = \cE(\varphi\, C(f_n))^{1/2} \leq 2 \|f\|_\infty \cE(\varphi f_n )^{1/2} + 3 \|f\|_\infty ^2 \cE(\varphi)^{1/2}.$$
 From this inequality  we infer
 \begin{align*} 
  \Eph(f_n) &= \cE(\varphi f_n) - \cE(\varphi f_n^2,\varphi) \\
  &\geq \cE(\varphi f_n) - \cE(\varphi f_n^2)^{1/2}\cE(\varphi)^{1/2}\\
  &\geq \cE(\varphi f_n)^{1/2} \left(\cE(\varphi f_n)^{1/2}  - 2\|f\|_\infty \cE(\varphi)^{1/2} \right) - 3\|f\|^2_\infty \cE(\varphi).
 \end{align*}
Therefore, the boundedness of $(\Eph(f_n))$ yields the boundedness of $(\cE(\varphi f_n))$ and this in turn yields the boundedness of $(\cE(\varphi f_n^2))$. Since $(f_n)$ is uniformly bounded by $\|f\|_\infty$ and since $\varphi \in L^2(m)$, Lebesgue's dominated convergence theorem implies $\varphi f_n \to \varphi f$ and $\varphi f_n^2 \to  \varphi f^2$ in $L^2(m)$. Therefore, the $L^2$-lower semicontinuity of $\cE$ yields $\varphi f, \varphi f^2 \in D(\cE)$, i.e., $f \in D(\Eph)$. From the boundedness of $(\cE(\varphi f_n^2))$ and the convergence $\varphi f_n^2 \to  \varphi f^2$ in $L^2(m)$ we obtain the $\cE$-weak convergence $\varphi f_n^2   \to \varphi f^2$, see Lemma~\ref{lemma:existence of a weakly convergent subnet}. This observation and the lower semicontinuity of $\cE$ on $L^2(m)$ amount to 
$$\Eph(f) = \cE(\varphi f) - \cE(\varphi f^2,\varphi) \leq \liminf_{n\to \infty} \cE(\varphi f_n) - \lim_{n\to \infty} \cE(\varphi f_n^2,\varphi) = \liminf_{n \to \infty} \Eph(f_n),$$
and  prove the desired lower semicontinuity. 

Having proven the closability of $\Eph$ on $L^0(m)$, we denote  its closure by $\Ep$. We now prove that $\Ep$ is Markovian. To this end, let $C:\R \to \R$ a normal contraction and let  $f \in D(\Ep)$. By Lemma~\ref{lemma:characterization closability} there exists an $\Eph$-Cauchy sequence $(f_n)$ in $D(\Eph)$ with $f_n \tom f$ and $\Eph(f_n) \to \Ep(f)$. Since $C$ is a normal contraction, we also have $C(f_n) \tom C(f)$. The lower semicontinuity of $\Ep$ and Lemma~\ref{lemma:properties of eph}~(b) yield $C(f_n) \in D(\Eph)$ and
$$\Ep(C(f)) \leq \liminf_{n\to \infty} \Ep(C(f_n)) = \liminf_{n\to \infty} \Eph(C(f_n)) \leq  \liminf_{n\to \infty} \Eph( f_n ) = \Ep(f).$$
Altogether, we have proven that $\Ep$ is an energy form. Its recurrence follows from the fact that $1 \in D(\Eph)$ and 
$$\Ep(1) = \Eph(1) = \cE(\varphi 1) - \cE(\varphi 1^2,\varphi) = 0.$$

(b): Since $\Ep$ is an extension of $\Eph$ and since $D(\Eph) \subseteq L^\infty(m)$, the inclusion $D(\Eph) \subseteq D(\Ep) \cap L^\infty(m)$ is trivial. Let $f \in  D(\Ep) \cap L^\infty(m)$ be given. According to Lemma~\ref{lemma:characterization closability}, there exists an $\Eph$-Cauchy sequence $(f_n)$ in $D(\Eph)$ with $f_n \tom f$ and $\Ep(f) = \lim_{n \to \infty}\Eph(f_n)$. The lower semicontinuity property that we proved above for $\Eph$ yields 
$$\Eph(f) \leq \liminf_{n\to\infty} \Eph(f_n) = \Ep(f) < \infty. $$
This shows $f \in D(\Eph)$.

(a): Let $f \in D(\Ee)$ and let $(f_n)$ be an $\cE$-Cauchy sequence with $f_n \tom f$. Since $D(\cE) \cap L^\infty(m)$ is dense in $D(\cE)$, see e.g. \cite[Theorem~1.4.2]{FOT}, we can choose the $(f_n)$ to be essentially bounded so that $f_n \in D(\Eph)$. The $L^0(m)$-lower semicontinuity of $\Ep$ and Lemma~\ref{lemma:properties of eph}~(d) yield
$$\Ep(f) \leq \liminf_{n\to \infty}\Ep(f_n) = \liminf_{n\to \infty}\Eph(f_n) \leq  \liminf_{n\to \infty} \cE(f_n) = \Ee(f).$$
This proves $\Ee \leq \Ep$. Now, let $\psi \in D(\cE)$ with $\varphi \leq \psi \leq 1$. For $f \in D(\cE_\psi)$ we choose an $\hat \cE_\psi$-Cauchy sequence $(f_n)$ with $f_n \tom f$. Lemma~\ref{lemma:properties of eph}~(c) yields $f_n \in D(\Eph)$ and  
$$\Ep(f) \leq \liminf_{n \to \infty} \Ep(f_n) = \liminf_{n \to \infty} \Eph(f_n) \leq \liminf_{n \to \infty} \hat \cE_\psi(f_n) = \cE_\psi(f).$$

(c): Let $\tilde \cE$ be a Dirichlet form that dominates $\cE$. For $f \in D(\tilde \cE_\varphi)$ we choose an $\hat{\tilde \cE}_\varphi$-Cauchy sequence $(f_n)$ with $f_n \tom f$ and $\hat{\tilde \cE}_\varphi(f_n) \to \tilde \cE_\varphi(f)$. The $L^0$-lower semicontinuity of $\Ep$ and Lemma~\ref{lemma:properties of eph}~(e) imply
$$\Ep(f) \leq \liminf_{n \to \infty} \Ep(f_n) = \liminf_{n \to \infty} \Eph(f_n) \leq \liminf_{n \to \infty}\hat{\tilde \cE}_\varphi(f_n) \leq \tilde \cE_\varphi(f). $$
This finishes the proof.
\end{proof}

\begin{definition}[Main part]
 The {\em main part of $\cE$} is defined by 
$$\Em:L^0(m) \to [0,\infty],\, \Em(f) := \sup\{ \Ep(f) \mid \varphi \in D(\cE) \text{ with } 0 \leq \varphi \leq 1\}.$$
The restriction of $\Em$ to $L^2(m)$ is called the {\em active main part of $\cE$} and is denoted by $\Em_a$.
\end{definition}

\begin{theorem}[Maximality of the main part]\label{theorem:maximality main part}
 The main part $\Em$ is a recurrent energy form that satisfies $\Ee \leq \Em$ and the active main part $\Em_a$ is a Dirichlet form that satisfies $\cE \leq \Em_a$.  Moreover, if $\tilde{\cE}$ is a Dirichlet form that dominates $\cE$, then also $\tilde{\cE}_{\rm e} \leq \Em$ and $\tilde{\cE} \leq \Em_a$.
\end{theorem}

\begin{proof}
As the supremum of lower semicontinuous functions on $L^0(m)$ the functional $\Em$ is lower semicontinuous. Moreover, the Markov property of $\Em$ and its recurrence follow from the Markov property  the recurrence of the $\Ep$. 

Next we prove that $\Em$ is a quadratic form. The homogeneity of $\Em$ follows easily from the homogeneity of $\Ep$. We let $I := \{\varphi \in D(\cE) \mid 0 \leq \varphi \leq 1\}$. For fixed $f \in L^0(m)$ the map $I \to [0,\infty],\, \varphi \mapsto \Ep(f)$ is monotone increasing and if $\varphi,\psi \in I$, then also $\varphi \vee \psi \in I$.  This monotonicity implies that for all $f,g \in L^0(m)$ we have
\begin{align*}
 \Em(f+g) + \Em(f-g) &= \sup_{\varphi \in I} \Ep(f+g) + \sup_{\varphi \in I}\Ep(f-g) \\
 &= \sup_{\varphi \in I} \left(  \Ep(f+g) +  \Ep(f-g) \right)\\
 &= \sup_{\varphi \in I} \left(  2\Ep(f) +  2\Ep(g) \right)\\
 &= 2 \sup_{\varphi \in I}   \Ep(f) +  2\sup_{\varphi \in I}\Ep(g) \\
 &= 2\Em(f)  + 2\Em(g).
\end{align*}
Therefore, $\Em$ is a quadratic form. 

Theorem~\ref{theorem:properties of concatenated forms} implies $\Ee \leq \Ep$ for all $\varphi \in I$ and therefore $\Ee \leq \Em$. As the $L^2$-restriction of an energy form, $\Em_a$ is clearly a Dirichlet form. The inequality $\cE \leq \Em_a$ follows from the statement for the extended Dirichlet form $\Ee$ and the identity $D(\Ee) \cap L^2(m) = D(\cE)$.

If $\tilde \cE$ is a Dirichlet form that dominates $\cE$,  Theorem~\ref{theorem:properties of concatenated forms} also shows  $\tilde \cE_\varphi \leq \Ep$ for all $\varphi \in I$. Moreover, what we have already proven applied to the form $\tilde \cE$ yields $\tilde \cE_{\rm e} \leq \tilde \cE_{\varphi}$ for any $\varphi \in I$. Combining these inequalities and taking a supremum over $\varphi$ shows $\tilde \cE_{\rm e} \leq \Em$. With the same argumentation as above for the form $\cE$, we also obtain $\tilde{\cE} \leq \Em_a$.
\end{proof}
\begin{remark}
 The previous theorem shows that $\Em_a$ is larger (in the sense of quadratic forms) than any Dirichlet form dominating $\cE$. Below we will also prove that   $\Em_a$  is a Dirichlet form that dominates $\cE$ so that $\Em_a$ is the maximal element in the cone of Dirichlet forms that dominate $\cE$.  
\end{remark}

For later purposes we note the following lower semicontinuity of $\Ep$ in the parameter $\varphi$.

\begin{lemma}\label{lemma:lsc in varphi}
 Let $(\varphi_n)$ a sequence in $D(\cE)$ and let $\varphi \in D(\cE)$ with $0\leq \varphi_n,\varphi \leq 1$. If $\varphi_n \to  \varphi$ with respect to $\|\cdot\|_\cE$, then for all $f \in L^0(m)$ we have
 $$\Ep(f) \leq \liminf_{n\to \infty} \cE_{\varphi_n}(f).$$
\end{lemma}
\begin{proof}
 We first prove the statement for $f \in L^\infty(m)$. Without loss of generality we can assume
 $$\lim_{n\to \infty} \cE_{\varphi_n}(f) = \liminf_{n\to \infty} \cE_{\varphi_n}(f) <  \infty.$$
 Since $D(\cE_{\varphi_n}) \cap L^\infty(m)  = D(\hat \cE_{\varphi_n})$, this implies $f \in D(\hat \cE_{\varphi_n})$ and $\cE_{\varphi_n}(f) = \hat \cE_{\varphi_n}(f)$ for each $n \in \N$. As in the proof of Theorem~\ref{theorem:properties of concatenated forms} we obtain the inequalities
 $$\cE(\varphi_n f^2)^{1/2} \leq 2 \|f\|_\infty \cE(\varphi_n f )^{1/2} + 3 \|f\|_\infty ^2 \cE(\varphi_n)^{1/2},$$
 and
 $$\cE(\varphi_n f)^{1/2} \left(\cE(\varphi_n f)^{1/2}  - 2\|f\|_\infty \cE(\varphi_n)^{1/2} \right) - 3\|f\|^2_\infty \cE(\varphi_n) \leq \hat \cE_{\varphi_n}(f).$$
 Hence, the boundedness of $(\cE(\varphi_n))$ and $(\hat \cE_{\varphi_n}(f))$ yields the boundedness of $(\cE(\varphi_n f))$ and  $(\cE(\varphi_n f^2))$. Since $\varphi_n f \to \varphi f$ in $L^2(m)$, the $L^2$-lower semicontinuity of $\cE$ yields $\varphi f \in D(\cE)$ and so $f \in D(\Eph)$ by Lemma~\ref{lemma:properties of eph}. Moreover, the $L^2$-convergence $\varphi_n f^2 \to \varphi  f^2$ and the $\cE$-boundedness of $(\varphi_n f^2)$ yields $\varphi_n f^2 \to \varphi  f^2$ $\cE$-weakly. Since also $\varphi_n \to \varphi$ in form norm, we obtain 
 $$\cE(\varphi_n f^2,\varphi_n) \to \cE(\varphi f^2,\varphi) \text{, as }n\to \infty.$$
The $L^2$-lower semicontinuity of $\cE$ and this observation yield
 $$\Ep(f) = \Eph(f) = \cE(\varphi f) - \cE(\varphi f^2,\varphi) \leq \liminf_{n \to \infty} \cE(\varphi_n f) - \lim_{n\to \infty} \cE(\varphi_n f^2,\varphi_n) = \liminf_{n\to \infty} \cE_{\varphi_n}(f).$$
 For general $f \in L^0(m)$ we consider $f^{(k)} := (f\wedge k) \vee (-k)$. Using the $L^0$-lower semicontinuity of $\Ep$,  what we have already proven for bounded functions and the Markov property of $\cE_{\varphi_n}$, we obtain
 $$\Ep(f) \leq \liminf_{k\to \infty} \Ep(f^{(k)}) \leq \liminf_{k\to \infty} \liminf_{n\to \infty}\cE_{\varphi_n}(f^{(k)}) \leq \liminf_{n\to \infty}\cE_{\varphi_n}(f).$$
 This finishes the proof.
\end{proof}

For a measurable subset $F$ of $X$ we let $D(\cE)_F = \{f \in D(\cE) \mid f\1_{X \setminus F} = 0\}$. Following \cite{AH}, we call an ascending sequence of measurable subsets $(F_n)_{n \in \N}$ of $X$ a {\em (measurable)   $\cE$-nest} if 
\begin{itemize}
 \item for each $n\in \N$ there exists $\varphi_n \in D(\cE)$ with $\varphi_n \geq \1_{F_n},$
 \item $\bigcup_{n \in \N} D(\cE)_{F_n}$ is dense in $D(\cE)$ with respect to $\|\cdot\|_\cE$.
\end{itemize}
According to \cite[Lemma~3.1]{AH} there always exits a nest. For the discussion in Section~\ref{section:regular} the following alternative formula for $\Em$ will be important. It shows that the supremum in the definition of $\Em$ can be taken along a suitable increasing sequence of functions. 
\begin{lemma} \label{lemma:alternative formula em}
 Let $(F_n)$ be an $\cE$-nest and let $(\chi_n)$ be an increasing sequence  in $D(\cE)$ with $\1_{F_n} \leq \chi_n \leq 1$. For all $f \in L^0(m)$ we have
 $$\Em(f) = \lim_{n \to \infty} \cE_{\chi_n}(f).$$
\end{lemma}
\begin{proof}
 It follows from the monotonicity of $\Ep$ in the parameter $\varphi$ that for any $f \in L^0(m)$ the limit $\lim_{n \to \infty} \cE_{\chi_n}(f)$ exists and satisfies 
 $$\lim_{n \to \infty} \cE_{\chi_n}(f) \leq \Em(f).$$
Thus, it suffices to prove the opposite inequality by showing that for any $\varphi \in D(\cE)$ with $0 \leq \varphi \leq 1$ and any $f \in L^0(m)$ we have    
 $$\Ep(f) \leq \lim_{n \to \infty} \cE_{\chi_n}(f).$$
 To this end, we use that $(F_n)$ is a nest and choose a sequence $(\varphi_k)$ in $\bigcup_n D(\cE)_{F_n}$ such that $\varphi_k \to \varphi$ with respect to $\|\cdot\|_\cE$. Without loss of generality we can assume $0 \leq \varphi_k \leq 1$. By the choice of $\varphi_k$ there exists $n_k$ such that $\varphi_k$ vanishes outside $F_{n_k}.$ Since $\chi_n \geq \1_{F_{n_k}}$, this implies $\varphi_k \leq \chi_{n_k}$ and we obtain
 $$\cE_{\varphi_k}(f) \leq \cE_{\chi_{n_k}}(f) \leq \lim_{n \to \infty} \cE_{\chi_n}(f).$$
 With this at hand, Lemma~\ref{lemma:lsc in varphi} yields
 $$\Ep(f) \leq \liminf_{k \to \infty} \cE_{\varphi_k}(f) \leq \lim_{n \to \infty} \cE_{\chi_n}(f).$$
 This finishes the proof.
\end{proof}

\begin{example}[Weighted manifolds]\label{example:mainfolds}
 For the notation used in this example we refer the reader to \cite{Gri}. Let $(M,g,\mu)$ a weighted Riemannian manifold and let $V \in L^1_{\rm loc}(\mu)$. We define the quadratic form $\D_0$ by letting $D(\D_0) = C_c^\infty(M)$ on which it acts by
 $$\D_0(f) = \int_M g(\nabla f, \nabla f) d\mu + \int_M V f^2 d\mu.$$
 The closure $\D$ of $\D_0$ in $L^2(\mu)$ is a regular Dirichlet form with domain $D(\D) = W_0^1(M) \cap L^2(V \cdot \mu)$. The domain of the main part of $\D$ is  $D(\D^{(M)}) = \{f \in L^2_{\rm loc}(\mu) \mid \nabla f \in \vec L^2(\mu)\}$ on which it acts by
 $$\D^{(M)}(f) = \int_M g(\nabla f, \nabla f) d\mu.$$
 Therefore, the active main part of $\D$ is the Dirichlet form of the weighted Laplacian on $(M,g,\mu)$ with Neumann boundary conditions at infinity, i.e., $D(\D^{(M)}_a) = W^1(M) =  \{f \in L^2(\mu) \mid \nabla f \in \vec L^2(\mu)\}$.
 \begin{proof}
 We first prove $\{f \in L^2_{\rm loc}(\mu) \mid \nabla f \in \vec L^2(\mu)\} \subseteq D(\D^{(M)})$ and the formula 
 $$\D^{(M)}(f) = \int_M g(\nabla f, \nabla f) d\mu$$
 for $f \in L^2_{\rm loc}(\mu)$ with $\nabla f \in L^2(\mu)$.  For $f \in L^2 (\mu)$ with $\nabla f \in  \vec L^2(\mu)$  and $\varphi \in C_c^\infty(M)$ with $0 \leq \varphi \leq 1$ we have $\varphi f \in L^2(m)$ and it follows from some distributional chain rule that $\nabla(\varphi f) \in \vec L^2(\mu)$. Since $\varphi f$ has compact support, \cite[Lemma~5.5]{Gri} implies $\varphi f \in W^1_0(M)$. If, additionally, $f \in L^\infty(\mu)$, we have also have $\varphi f \in L^2(V \cdot \mu)$ and therefore $\varphi f \in D(\D)$. In this case, another application of the chain rule shows
 $$\D_\varphi(f) = \D(\varphi f) - \D(\varphi f^2,\varphi) = \int_M \varphi^2 g(\nabla f,\nabla f) d\mu.$$
 Let now $f \in L^2_{\rm loc}(\mu)$ with $\nabla f \in  \vec L^2(\mu)$ and for $n \in \N$ set $f^{(n)} = (f \wedge n) \vee(-n)$.  Since $\D_\varphi$ is an energy form, we have $\D_\varphi(f) = \lim_n \D_\varphi(f^{(n)})$. Moreover, $\nabla f^{(n)} \to \nabla f$ in $L^2_{\rm loc}(\mu)$. Therefore, the above  identity extends to unbounded functions as well. Letting $\varphi \nearrow 1$ and using Lemma~\ref{lemma:alternative formula em} shows $\{f \in L^2_{\rm loc}(\mu) \mid \nabla f \in L^2(\mu)\}$ and the desired formula.
 
 It remains to prove $D(\D^{(M)}) \subseteq  \{f \in L^2_{\rm loc}(\mu) \mid \nabla f \in \vec L^2(\mu)\}$. Let $f \in D(\D^{(M)}) \cap L^\infty(m)$. Then for each $\varphi \in C_c^\infty(M)$ with $0 \leq \varphi \leq 1$ we have $\varphi f \in \D$ and therefore $\nabla(\varphi f) \in \vec L^2(\mu)$. This shows $\nabla f \in \vec L^2_{\rm loc}(\mu)$. Similar computations as above yield
 $$\D_\varphi(f) = \int_M \varphi^2 g(\nabla f, \nabla f) d\mu$$
 and after letting $\varphi \nearrow 1$ we arrive at $\nabla f \in \vec L^2(\mu)$. For unbounded $f \in D(\D^{(M)})$ what we have already proven and Lemma~\ref{lemma:bounded approximation} yields
 $$\sup_n \int_M g(\nabla f^{(n)},\nabla f^{(n)}) d \mu =  \sup_n \D^{(M)}(f^{(n)}) = \D^{(M)}(f) < \infty.$$
 With this at hand,  a local Poincaré inequality yields $f \in L^2_{\rm loc}(\mu)$ and $\nabla f \in L^2(\mu)$ (for more details on this conclusion see \cite[Proof of Proposition~7.1]{HKLMS}).
 \end{proof}
 \end{example}

\begin{example}[Weighted graphs]
For  a background on Dirichlet forms associated with weighted graphs we refer the reader to \cite{KL}. Let $X$ a countably infinite set. Any function $m:X \to (0,\infty)$ induces a measure of full support on $X$ by letting 
$$m(A) = \sum_{x \in A}m(x).$$
A weighted graph over $X$ is a pair $(b,c)$ consisting of a function $c:X \to [0,\infty)$ and a function $b:X \times X \to [0,\infty)$ with
\begin{itemize}
 \item[(b0)] $b(x,x) = 0$  for all $x \in X$, 
 \item[(b1)] $b(x,y) = b(y,x)$ for all $x,y \in X$,
 \item[(b2)] $\sum_{y\in X} b(x,y) < \infty$ for all $x \in X$.
\end{itemize}
The value $b(x,y)$ is considered to be the edge weight of the edge $(x,y)$. We define the form $Q_0$ by letting  $D(Q_0) = C_c(X)$, the finitely supported functions on $X$, on which $Q_0$ acts by
$$Q_0(f) = \sum_{x,y\in X} b(x,y)(f(x)-f(y))^2 + \sum_{x \in X} c(x) f(x)^2. $$
The closure  $Q$ of $Q_0$ on $\ell^2(X,m)$ is a regular Dirichlet form. Then, the domain of the main part of $Q$ is  
$$D(Q^{(M)}) = \{f :X \to \R \mid \sum_{x,y\in X} b(x,y)(f(x)-f(y))^2 < \infty \}$$
on which it acts by
 $$Q^{(M)}(f) = \sum_{x,y\in X} b(x,y)(f(x)-f(y))^2.$$
 \begin{proof}
  It follows from the definitions  and a simple approximation argument that for $\varphi \in C_c(X)$ with $0\leq \varphi \leq 1$ we have $D(Q_\varphi) = \{f:X\to \R\}$ on which it acts by
 $$Q_\varphi(f) = \sum_{x,y\in X} b(x,y) \varphi(x) \varphi(y) (f(x) - f(y))^2.$$
 This observation and Lemma~\ref{lemma:alternative formula em} imply the claim after letting $\varphi \nearrow 1$. 
 \end{proof}
\end{example}

\subsection{The killing part and the reflected Dirichlet form}

The given examples show that in general $\Em$ is not an extension of $\Ee$ and   $\Em_a$ is not an extension of $\cE$.  However, on $D(\cE)$ the form $\cE$  is  a perturbation of $\Em$ by a monotone quadratic form. In this subsection we (prove and) employ this monotonicity to construct the killing part of $\cE$ and discuss its properties.

We define the {\em preliminary killing part of $\cE$} by 
$$\Ekh:L^0(m) \to [0,\infty],\, \Ekh(f) := \begin{cases}
                                            \cE(f) - \Em(f) &\text{if } f \in D(\cE), \\
                                            \infty &\text{else.}
                                           \end{cases}
$$
It follows from the considerations in the previous subsection that $\Ekh$ is a nonnegative quadratic form with domain $D(\cE)$. In contrast to $\Eph$ it is in general not closable on $L^0(m)$, see \cite[Remark~3.50]{Schmi}.  Instead, we will use the following monotonicity property of $\Ekh$ to  extend it to a larger domain.

\begin{lemma}\label{lemma:monotonicity of ekh}
 For $f,g \in D(\cE)$ the inequality $|f| \leq |g|$ implies $\Ekh(f) \leq \Ekh(g)$.
\end{lemma}
\begin{proof}
 Let $f,g \in D(\cE) \cap L^\infty(m)$ with $|f|\leq |g|$ be given. For $\varepsilon > 0$ we set $f_\varepsilon := f - (f \wedge \varepsilon) \vee (-\varepsilon)$, $g_\varepsilon := g - (g \wedge \varepsilon) \vee (-\varepsilon)$ and $\varphi_\varepsilon := (\varepsilon^{-1}g) \wedge 1$ such that $\varphi_\varepsilon = 1$ on $\{f_\varepsilon > 0\}$ and $\{g_\varepsilon > 0\}$.    According to \cite[Theorem~1.4.2]{FOT}, we have $f_\varepsilon \to f$ and $g_\varepsilon \to g$ with respect to $\cE$, as $\varepsilon \to 0+$. Since $\Ekh$ is continuous with respect to $\cE$ convergence, this implies
 \begin{align*}
  \Ekh(g) - \Ekh(f) &= \lim_{\varepsilon \to 0+} \left( \Ekh(g_\varepsilon) - \Ekh(f_\varepsilon)\right)\\
  &=  \lim_{\varepsilon \to 0+} \left( \cE(g_\varepsilon) - \Em(g_\varepsilon) - \cE(f_\varepsilon) +  \Em(f_\varepsilon)\right)\\
  &= \lim_{\varepsilon \to 0+}\lim_{ \varphi \in D(\cE),\, \varphi \nearrow 1}  \left( \cE(g_\varepsilon) - \Ep(g_\varepsilon) - \cE(f_\varepsilon) +  \Ep(f_\varepsilon)\right).
 \end{align*}
 Moreover, for any $\varphi \in D(\cE)$ with $\varphi_\varepsilon \leq \varphi \leq 1$  we have
$$
 \cE(g_\varepsilon) - \Ep(g_\varepsilon) - \cE(f_\varepsilon) +  \Ep(f_\varepsilon) =  \cE(g_\varepsilon) - \Eph(g_\varepsilon) - \cE(f_\varepsilon) +  \Eph(f_\varepsilon)= \cE(g_\varepsilon^2 - f_\varepsilon^2,\varphi).
$$
Since $g_\varepsilon^2 - f_\varepsilon^2 \geq 0$ and $\varphi = 1$ on  $\{g_\varepsilon^2 - f_\varepsilon^2  > 0\}$, for any $\alpha > 0$ we obtain
$$\cE(\varphi) = \cE\left( \left(\varphi +\alpha(g_\varepsilon^2 - f_\varepsilon^2) \right) \wedge 1 \right) \leq \cE(\varphi) + 2\alpha \cE(g_\varepsilon^2 - f_\varepsilon^2,\varphi) + \alpha^2 \cE(g_\varepsilon^2 - f_\varepsilon^2). $$
Rearranging this inequality and letting $\alpha \to 0+$ yields $\cE(g_\varepsilon^2 - f_\varepsilon^2,\varphi) \geq 0$ and finishes the proof.
\end{proof}

\begin{definition}[Killing part]
 The {\em killing part} of $\cE$ is defined by
$$\Ek:L^0(m) \to [0,\infty],\, \Ek(f):= \begin{cases}
                        \sup\{\Ekh(g) \mid g \in D(\cE) \text{ with } |g| \leq |f|\} &\text{if } f \in D(\Em),\\
                        \infty &\text{else}.
                       \end{cases}$$
\end{definition}

It follows from the previous lemma that $\Ek$ is an extension of $\Ekh$. In order to prove further properties of $\Ek$ we need the following alternative formula.

\begin{lemma} \label{lemma:alternative formula ek}
 For all $f \in D(\Em)$ we have
 $$\Ek(f) = \sup\{\Ekh(\varphi f^{(n)}) \mid n \in \N, \varphi \in D(\cE) \text{ with } 0 \leq \varphi \leq 1\},$$
 where $f^{(n)} = (f \wedge n) \vee (-n)$.
\end{lemma}
\begin{proof}
For $f \in D(\Em)$ we let 
$$k(f) := \sup\{\Ekh(\varphi f^{(n)}) \mid n \in \N, \varphi \in D(\cE) \text{ with } 0 \leq \varphi \leq 1\}.$$
Since $D(\cE) \cap L^\infty(m)$ is an algebraic ideal in $D(\Em) \cap L^\infty(m)$ and $|\varphi f^{(n)}| \leq |f|$, the monotonicity of $\Ekh$ on $D(\cE)$ implies $k(f) \leq \Ek(f)$. 

In order to deduce the opposite inequality, we use the monotonicity of $\Ekh$ and that $D(\cE)$ is a lattice to choose an increasing sequence of nonnegative functions $(f_l)$  in $D(\cE)$ such that $\Ek(f) = \sup_l \Ekh(f_l)$.   Since $f_l^{(n)}\overset{n\to \infty}{\longrightarrow} f_l$  with respect to $\cE$, see e.g. \cite[Theorem~1.4.2]{FOT}, and since $\Ekh$ is continuous with respect to $\cE$-convergence, we can further assume the $(f_l)$ to be bounded. For $\varepsilon > 0$ we let $\varphi_{l,\varepsilon} := f_l/(\varepsilon + f_l)$. It follows from Lemma~\ref{lemma:contraction properties} that $\varphi_{l,\varepsilon} \in D(\cE)$ and that $\varphi_{l,\varepsilon} f_l$ satisfies
$$\cE(\varphi_{l,\varepsilon}f_l) \leq \cE(f_l).$$
This inequality together with $\varphi_{l,\varepsilon}f_l \tom f_l$ as $\varepsilon \to 0 +$ yields $\varphi_{l,\varepsilon}f_l \to f_l$ with respect to $\cE$, see Lemma~\ref{lemma:existence of a weakly convergent subnet}. Since $\Ekh$ is monotone and continuous with respect to $\cE$, we obtain
$$\Ekh(f_l) = \lim_{\varepsilon \to 0+} \Ekh(\varphi_{l,\varepsilon} f_l) \leq k(f).$$
This finishes the proof.
\end{proof}

The following theorem summarizes  the  properties of $\Ek$.

\begin{theorem} \label{lemma:properties of ek}
The functional $\Ek$ is a monotone quadratic form on $L^0(m)$ that extends $\Ekh$. Its domain satisfies $D(\Ek) \subseteq D(\Em)$. If $(f_n)$ is an $\Em$-bounded sequence in $D(\Em)$ with $f_n \tom f$, then
$$\Ek(f) \leq \liminf_{n\to \infty}\Ek(f_n).$$
\end{theorem}
\begin{proof}
 The monotonicity of $\Ek$ and that it extends $\Ekh$ is immediate from its definition and the monotonicity of $\Ekh$. Next we show that $\Ek$ is a quadratic form.  For   $\lambda \in \R$ and $f \in D(\Em)$ the identity 
 $$\Ek(\lambda f) = |\lambda|^2\Ek(f)$$ follows easily from the definition of $\Ek$ and the corresponding statement for $\Ekh$. We employ Lemma~\ref{lemma:alternative formula ek} to  verify the parallelogram identity. Let $f,g \in D(\Em)$,  $\varphi \in D(\cE)$ with $0 \leq \varphi \leq 1$ and   $n \in \N$. The inequalities $|f + g| \geq |\varphi(f^{(n)} + g^{(n)})|$ and $|f - g| \geq |\varphi(f^{(n)} - g^{(n)})|$, the monotonicity of $\Ek$ and the fact that $\Ekh$ is a quadratic form yield
 \begin{align*}
  \Ek(f + g) + \Ek(f-g) &\geq \Ekh(\varphi(f^{(n)} + g^{(n)})) + \Ekh(\varphi(f^{(n)} - g^{(n)}))\\
  &= 2 \Ekh(\varphi g^{(n)}) + 2 \Ekh(\varphi f^{(n)}).
 \end{align*}
 Lemma~\ref{lemma:alternative formula ek} implies
 $$\Ek(f + g) + \Ek(f-g) \geq 2\Ek(f) + 2\Ek(g).$$
 According to Lemma~\ref{lemma:characterization quadratic forms} this is inequality and the homogeneity are sufficient for proving that $\Ek$ is a quadratic form. 

 It remains to show the  statement on lower semicontinuity. To this end, let $(f_n)$ an $\Em$-bounded sequence in $D(\Em)$ with $f_n \tom f$. Since $\Em$ is closed, we have $f \in D(\Em)$. Therefore, Lemma~\ref{lemma:alternative formula ek} implies that it suffices to show 
 $$\Ekh(\varphi f^{(l)} ) \leq \liminf_{n\to \infty}\Ekh(\varphi f_n^{(l)})$$
 for every $l \in \N$ and $\varphi \in D(\cE)$ with $0 \leq \varphi \leq 1$. The monotonicity of $\Ekh$ and Lemma~\ref{lemma:algebraic properties} yield
 $$\cE(\varphi f_n^{(l)}) = \Em(\varphi f_n^{(l)}) + \Ekh(\varphi f_n^{(l)})  \leq 2 l^2 \Em(\varphi) + 2 \Em(f_n) +  l^2 \Ekh(\varphi).$$
 Since $(f_n)$ is $\Em$-bounded, we obtain that $(\varphi f_n^{(l)})$ is  $\cE$-bounded and $\Em$-bounded. Furthermore, we have the $L^2$-convergence $\varphi f_n^{(l)}   \to \varphi f^{(l)}$, as $n\to \infty$, and so Lemma~\ref{lemma:existence of a weakly convergent subnet} shows  $\varphi f_n^{(l)}   \to \varphi f^{(l)}$ $\cE$-weakly and $\Em$-weakly, i.e., $\varphi f_n^{(l)}   \to \varphi f^{(l)}$ $\Ekh$-weakly. Quadratic forms are lower semicontinuous with respect to weak convergence of the induced bilinear form. Therefore, we arrive at the desired inequality. 
 \end{proof}

 \begin{example}[Weighted manifolds, continued]
 We prove that the domain of $\D^{(k)}$ is given by $D(\D^{(k)}) = \{f \in L^2(\mu) \mid \nabla f \in \vec L^2(\mu) \text{ and } \int_M V f^2 d \mu <  \infty \}$ on which it acts by
 $$\D^{(k)}(f) = \int_M V f^2 d\mu.$$
 \begin{proof}
  If follows from what we have already proven on $\D^{(M)}$ that 
  $$\hat \D^{(k)}(f) = \D(f) - \D^{(M)}(f) = \int_M V f^2 d\mu$$
  for $f \in D(\D)$. From this identity it easily follows from the monotone convergence theorem that for $f \in D(\D^{(M)}) = \{g \in L^2_{\rm loc}(\mu) \mid \nabla \vec L^2(\mu)\}$   the supremum
  $$ \sup \{ \D^{(k)}(h) \mid  h \in D(\D) \text{ with } |h| \leq |f|\} $$
  is finite  if and only if $\int_M V f^2 d\mu <  \infty$ in which case it equals $\int_M V f^2 d\mu$.  This finishes the proof.
 \end{proof}
 \end{example}

\begin{example}[Weighted graphs, continued]
 With the same arguments as for manifolds it follows that the domain of the killing part of $Q$ satisfies
 $$D(Q^{(k)}) = \{f :X \to \R\mid \sum_{x,y} b(x,y) (f(x)- f(y))^2 + \sum_{x \in X} c(x) f(x)^2 < \infty\}$$
 on which it acts by
 $$Q^{(k)}(f) = \sum_{x\in X} c(x) f(x)^2.$$
\end{example}

\subsection{The (active) reflected Dirichlet form and its maximality} In this subsection we use the main part and the killing part to introduce the reflected Dirichlet form. We show that it is a Silverstein extension and discuss its maximality among all Silverstein extensions. As a main result we obtain that the active main part is the maximal form among all Dirichlet forms that dominate the given form. In contrast, we give an example of a Dirichlet form which does not possess a maximal Silverstein extension.

\begin{definition}[Reflected Dirichlet form]
 Let $\cE$ be a Dirichlet form. The {\em reflected form of $\cE$} is  $\Er := \Em + \Ek$. The restriction of $\Er$ to $L^2(m)$ is called the {\em active reflected Dirichlet form of $\cE$} and is denoted by $\Era$. 
\end{definition}

 The results of the previous subsection accumulate in the following theorem.

\begin{theorem}\label{theorem:er is a silverstein extension}
 $\Er$ is an energy form that extends $\Ee$ and the Dirichlet form $\Era$ is a Silverstein extension of $\cE$.
\end{theorem}
\begin{proof}
 The closedness of $\Er$ on $L^0(m)$ follows from the closedness of $\Em$ on $L^0(m)$ and Lemma~\ref{lemma:properties of ek}. That it is Markovian is a consequence of the Markov property of $\Em$ and the monotonicity of $\Ek$. Furthermore, by the very definition of $\Ekh$ the forms $\Er$ and $\cE$ agree on $D(\cE)$.
 
 Next we prove that $\Er$ extends $\Ee$. To this end, let $f \in D(\Ee)$ and let $(f_n)$ an $\cE$-Cauchy sequence in $D(\cE)$ with $f_n \tom f$. The form $\Er$ is closed and so its domain $D(\Er)$ is complete with respect to the form topology, cf. Lemma~\ref{lemma:completeness v.s. lower semicontinuity} and the discussion preceding it. Since $\Ee$ and $\Er$ agree on $D(\cE)$, the sequence $(f_n)$ is a Cauchy sequence in the form topology of $\Er$. By completeness this implies $f \in D(\Er)$ and 
 $$\Er(f) = \lim_{n \to \infty} \Er(f_n) = \lim_{n \to \infty} \cE(f_n) = \Ee(f).$$

 For proving that $\Era$ is a Silverstein extension of $\cE$ we show that $D(\cE) \cap L^\infty(m)$ is an algebraic ideal in $D(\Era) \cap L^\infty(m)$. Let $\varphi \in D(\cE) \cap L^\infty(m)$ and $f \in D(\Era) \cap L^\infty(m)$. Since $D(\cE)$ is a lattice it suffices to consider $0 \leq \varphi  \leq 1$. By definition we have
 $$D(\Era) \cap L^\infty(m)  \subseteq D(\Em) \cap L^\infty(m) \subseteq D(\Ep)\cap L^\infty(m) = \{g \in L^\infty(m) \mid \varphi g \in D(\cE)\},$$
where we used Theorem~\ref{theorem:properties of concatenated forms} for the last equality. This shows $f\varphi \in D(\cE)$ and finishes the proof.
 \end{proof}

New approaches towards defining reflected forms naturally raise the question  whether  these approaches coincide or not. The following maximality statement about the main part of partially settles this issue. It is the main structural result about the active main part.
 
\begin{theorem}\label{theorem:maximality active reflected form}
Let $\cE$ be a Dirichlet form. Its active main part $\Em_a$ is the maximal Dirichlet form that dominates $\cE$. In particular, if $\Ek= 0$, then $\Era = \Em_a$ is the maximal Silverstein extension of $\cE$.
\end{theorem}
\begin{proof}
We proved in Theorem~\ref{theorem:maximality main part} that $\Em_a$ is larger than any Dirichlet form that dominates $\cE$. Hence, it only remains to prove that  $\Em_a$ dominates $\cE$. To this end, we employ Lemma~\ref{lemma:characterization of domination}. It follows along the same lines as at the end of the proof of Theorem~\ref{theorem:er is a silverstein extension} that $D(\cE) \cap L^\infty(m)$ is an algebraic ideal in $D(\Em_a) \cap L^\infty(m)$. For nonnegative $f,g \in D(\cE)$ we infer
$$\cE(f,g) - \Em_a(f,g) = \Ekh(f,g) \geq 0$$
from Lemma~\ref{lemma:monotonicity of ekh} and Lemma~\ref{lemma:characterization monotonicity}, which we can apply since $D(\Em)$ is a lattice, see Lemma~\ref{lemma:algebraic properties}.  The 'In particular' part follows from the identity $\Em_a = \Era$ if $\Ek=0$. This finishes the proof.
\end{proof}

\begin{remark}
 It is proven in \cite[Theorem~6.6.9]{CF} for quasi-regular Dirichlet forms that the active reflected Dirichlet form constructed there is always the maximal Silverstein extension of the given form. While the proof of \cite[Theorem~6.6.9]{CF} contains a mistake (see also the counterexample below), it is true when the given form has no killing (in the sense of the Beurling-Deny decomposition of quasi-regular forms). Since the maximal Silverstein extension is unique, in this case we obtain from Theorem~\ref{theorem:maximality active reflected form} that the active reflected Dirichlet form in the sense of \cite{CF} coincides with the active reflected Dirichlet form (and the active main part) in our sense. With the theory developed in this paper it is also not so difficult to prove that both approaches to reflected Dirichlet spaces agree for all quasi-regular Dirichlet forms. Our main part corresponds to the sum of the extended strongly local part and the extended jump part (in the sense of the Beurling-Deny decomposition) and our killing part coincides with the killing part of the Beurling-Deny decomposition.  However, spelling out all the definitions and details is a bit lengthy and so we leave the details as an exercise to the reader. For weighted manifolds and weighted graphs we have given the necessary arguments in the examples. 
 
 In \cite{Rob} there is another construction for maximal Dirichlet forms that dominate a given inner regular local Dirichlet form. It follows from  \cite[Corollary~4.3]{Rob} that our form $\Era$ coincides with the form $\cE_m$ constructed in \cite{Rob} if $\cE$ has no killing. 

 \end{remark}

\begin{remark}
Given a Dirichlet form $\cE$ we have proven that $\Em_a$ is the maximal element (in the sense of quadratic forms) in the cone of Dirichlet forms that dominate $\cE$. It would be interesting to know whether it is also the maximal element among all quadratic forms that dominate $\cE$. Note that any form that dominates $\cE$ has a positivity preserving resolvent and therefore satisfies the first Beurling-Deny criterion, while it need not satisfy the second Beurling-Deny criterion.  
\end{remark}

In general $\Er$ is not the maximal Silverstein extension of $\cE$. In fact, the following example shows that there are regular Dirichlet forms that do not possess a maximal Silverstein extension. In particular, it is a counterexample to the (wrong) claims of \cite[Theorem~5.1]{Kuw} and \cite[Theorem~6.6.9]{CF}, which state for quasi-regular Dirichlet forms that the active reflected Dirichlet form is always the maximal Silverstein extension. 

We let $\lambda$ be the Lebesgue measure on all Borel subsets of the open interval $(-1,1)$ and let  
$$W^1((-1,1)) = \{f \in L^2(\lambda) \mid f' \in L^2(\lambda)\}$$
the usual Sobolev space of first order. It is folklore that $W^1((-1,1))$ equipped with the norm
$$ \|f\|_{W^1} := \sqrt{\|f\|_2^2 + \|f'\|_2^2} $$
is a Hilbert space, which continuously embeds into $(C([-1,1]),\|\cdot\|_\infty)$. In particular,  any function in $W^1((-1,1))$ can be uniquely extended to the boundary points $-1$ and $1$. Then
$$W_0^1((-1,1)) = \{ f\in W^1((-1,1)) \mid f(-1) = f(1) = 0\} $$
coincides with the closure of $C_c((-1,1)) \cap W^1((-1,1))$ in the space $(W^1((-1,1)),\|\cdot\|_{W^1})$.

\begin{proposition}\label{proposition:counterexample}
The Dirichlet form 
$$\cE:L^2(\lambda) \to [0,\infty], \cE(f) = \begin{cases} 
                                                        \int_{-1}^1|f'|^2 d\lambda + f(0)^2 & \text{ if } f \in W_0^1((-1,1))\\
                                                        \infty &\text{ else}
                                                       \end{cases}
$$
is regular and does not possess a maximal Silverstein extension. 
\end{proposition}
\begin{proof}
 The regularity of $\cE$ follows from the properties of the Sobolev space $W_0^1((-1,1))$. Assume that $\ow{\cE}$ were the maximal Silverstein extension of $\cE$. The Dirichlet forms 
  $$\cE_1:L^2(\lambda) \to [0,\infty],\cE_1(f) = \begin{cases} 
                                                        \int_{-1}^1|f'|^2 d \lambda + (f(0)-f(1))^2 & \text{ if } f \in W^1((-1,1))\\
                                                        \infty &\text{ else}
                                                       \end{cases}
$$
 and
 $$\cE_2:L^2(\lambda) \to [0,\infty],\cE_2(f) = \begin{cases} 
                                                        \int_{-1}^1|f'|^2 d \lambda + f(0)^2 & \text{ if } f \in W^1((-1,1))\\
                                                        \infty &\text{ else}
                                                       \end{cases}
$$
are  Silverstein extensions of $\cE$, as obviously $D(\cE)$ is an order ideal in $D(\cE_1)$ and $D(\cE_2)$. The maximality of $\ow{\cE}$ implies $W^1((-1,1)) = D(\cE_1) \subseteq (\ow{\cE})$ and 
$$\ow{\cE}(1) \leq \cE_1(1) = 0.$$
We choose a function $f \in W_0^1((-1,1))$ with $f(0) = 1$. The equality $\ow{\cE}(1) = 0$ and the maximality of $\ow{\cE}$ yield 
\begin{align*}
   \int_{-1}^1|f'|^2 d \lambda + 1 = \cE(f) = \ow{\cE}(f) = \ow{\cE}(f - 1) \leq \cE_2(f-1) = \int_{-1}^1|f'|^2 d\lambda,
\end{align*}
a contradiction. This shows that $\cE$ does not possess a maximal Silverstein extension.  
\end{proof}

\section{Reflected regular Dirichlet forms}\label{section:regular}

In this we employ our new construction of the reflected form to prove that for regular forms the space of continuous functions is dense in the domain of the active reflected Dirichlet form. We then briefly sketch how this observation leads to a construction of the associated reflected process on a compactification of the underlying space.

Let $(X,d)$ a separable locally compact metric space and let $m$ a Radon measure of full support on $X$. By $C_c(X)$ we denote the space of compactly supported continuous functions on $X$. Recall that a Dirichlet form $\cE$ on $L^2(m)$ is called {\em regular} if $C_c(X) \cap D(\cE)$  is dense 
\begin{itemize}
 \item in $D(\cE)$ with respect to the form norm,
 \item in $C_c(X)$ with respect to the uniform norm.
\end{itemize}
In what follows $\cE$ is a fixed regular Dirichlet form on the metric measure space $(X,d,m)$.

\begin{lemma}[Regular partitions of unity]\label{lemma:regular partitions of unity}
For every  sequence of  relatively compact  open sets $(G_n)$ with $\overline{G_n} \subseteq G_{n+1}$, for all $n \in \N$, and $\bigcup_n G_n = X$ there exist functions $\psi_n \in D(\cE)\cap C_c(X)$ with $0\leq \psi_n \leq 1$, $\supp \psi_n \subseteq G_{n + 1} \setminus \overline{G_{n-1}}$ and 
$$\sum_{n = 1}^\infty \psi_n = 1.$$
\end{lemma}
\begin{proof}
 Set $\Omega_n := G_{n+1} \setminus \overline{G_{n-1}}$. Since $\cE$ is regular, there exist functions $\varphi_n \in D(\cE) \cap C_c(X)$ with $0 \leq \varphi_n \leq 1$, $\varphi_n = 1$ on $G_n$ and $\supp \varphi_n \subseteq \Omega_n$, cf. \cite[Exercise~1.4.1]{FOT}. We let 
 $$\varphi:= \sum_{n = 1}^\infty \varphi_n \text{ and } \psi_n := \frac{\varphi_n}{\varphi}. $$
 Since $(\Omega_n)$ is a locally finite cover of $X$, the function $\varphi$ is bounded and continuous. Moreover, it satisfies $\varphi \geq 1$. Therefore, $\psi_n \in C_c(X)$ and 
 $$\sum_{n=1}^\infty \psi_n = 1.$$
 It remains to prove $\psi_n \in D(\cE)$. The property $\supp \varphi_n \subseteq \Omega_n$ yields
 $$\psi_n = \begin{cases}
             \frac{\varphi_n}{\varphi_{n-1} + \varphi_n + \varphi_{n+1}} &\text{on } \{\varphi_n > 0\},\\
             0&\text{on } \{\varphi_n = 0\}.
            \end{cases}
$$
 Hence, $\psi_n$ is of the form $g^{-1}f$ with $f,g \in D(\cE)$, $0 \leq f \leq g$ and $g \geq 1 $ on $\{f > 0\}$, where we use the convention $g^{-1}f = 0$ on $\{f = 0\}$ . Such functions satisfy
 $$\left|\frac{f(x)}{g(x)}\right| \leq |f(x)| \text{ and } \left|\frac{f(x)}{g(x)} - \frac{f(y)}{g(y)}\right| \leq |f(x) - f(y)| + |g(x) - g(y)|.$$
 Therefore, Lemma~\ref{lemma:contraction properties} yields $g^{-1}f \in D(\cE)$ and the claim is proven.
\end{proof}

The following lemma is a variant of \cite[Corollary~2.3.1]{FOT} and we include a proof for the convenience of the reader. For a discussion of the related question of Kac regularity we  refer  to the recent \cite{Wir}.
\begin{lemma}\label{lemma:help domains}
 Let $\Omega$ relatively compact open. Then $\{\psi \in D(\cE) \cap C_c(X) \mid \supp \psi \subseteq \Omega\}$ is dense in $ \{f \in D(\cE)  \mid \text{there ex. compact } K\subseteq \Omega \text{ with } f 1_{X \setminus K} = 0\}$ with respect to the form norm. 
\end{lemma}
 \begin{proof}
  Let $f\in D(\cE)$  and $K \subseteq \Omega$ compact such that $f \1_{X \setminus K} = 0$. Without loss of generality we assume $0 \leq f \leq 1$. Since $\cE$ is regular, there exists $\psi \in D(\cE) \cap C_c(X)$ with $\psi = 1$ on $K$ and $\psi = 0$ on $X \setminus \Omega$, cf. \cite[Exercise~1.4.1]{FOT}. Let now $(f_n)$ a sequence in $D(\cE)\cap C_c(X)$ with $f_n \to f$ with respect to the form norm and set $g_n := \psi \cdot (f_n \vee 0) \wedge 1 $.  Then $(g_n)$ is an $\cE$-bounded sequence in $D(\cE) \cap C_c(X)$ with $\supp g_n \subseteq \Omega$ that converges in $L^2(m)$ towards $f$. By Lemma~\ref{lemma:existence of a weakly convergent subnet} it also converges $\cE$-weakly. Now the Banach-Saks theorem implies the desired statement.  
 \end{proof}
The following density statements are the main insights of this section. Its proof is based on the ideas of \cite{MeSe}.
\begin{theorem}\label{theorem:continous functions are dense}
\begin{itemize}
 \item[(a)]  $D(\Em) \cap C(X)$ is dense in $D(\Em)$ with respect to $\Em$.
 \item[(b)] $D(\Er) \cap C(X)$ is dense in $D(\Er)$ with respect to $\Er$.
 \item[(c)] $D(\Em_a) \cap C(X)$ is dense in $D(\Em_a)$ with respect to the form norm of $\Em_a$.
\item[(d)] $D(\Era) \cap C(X)$ is dense in $D(\Era)$ with respect to the form norm of $\Era$. 
\end{itemize}

\end{theorem}

\begin{proof}

 Starting with $f \in D(\Em)$, we construct a continuous function $g$ which is close to $f$ with respect to $\Em$. We then show that under the additional assumption $f \in D(\Er)$, the previously constructed function $g$ is also close to $f$ with respect $\Er$. Similarly, we treat the $L^2$-statement.

 (a): Since bounded functions are dense in domains of energy forms, see Lemma~\ref{lemma:bounded approximation}, it suffices to prove that each function $f \in D(\Em) \cap L^\infty(m)$ can be approximated by continuous ones. Let $(G_n)$ a   sequence of  relatively compact  open sets with   $\overline{G_n} \subseteq G_{n+1}$, for all $n \in \N$, and $\bigcup_n G_n = X$,  and let $(\psi_n)$ in $D(\cE) \cap C_c(X)$ a corresponding partition of unity as in Lemma~\ref{lemma:regular partitions of unity}. We set $\Omega_n := G_{n+1} \setminus \overline{G_{n-1}}$. 
 
 Let $\varepsilon > 0$. By construction $D(\cE) \cap L^\infty(m)$ is an algebraic ideal in $D(\Em) \cap L^{\infty}(m)$, cf. proof of Theorem~\ref{theorem:er is a silverstein extension}. Hence, it follows that  $f \psi_n \in D(\cE)$ with $f \psi_n = 0$ a.e. on $(\supp \psi_n)^c$. By definition  $\supp \psi_n $ is a compact subset of $\Omega_n$ and so Lemma~\ref{lemma:help domains} yields a function $g_n \in C_c(X) \cap D(\cE)$ with $\supp g_n \subseteq \Omega_n$,  
 $$\cE(f \psi_n - g_n) \leq \frac{\varepsilon}{4^{n}} \text{ and } \|f\psi_n - g_n\|_2 \leq \frac{\varepsilon}{2^{n}}.$$
 Moreover,  $g_n$ can be chosen to also satisfy $\|g_n\|_\infty \leq \|f\|_\infty$, see Lemma~\ref{lemma:bounded approximation}. Let $g := \sum_{n = 1}^\infty g_n$. Since $(\Omega_n)$ is a locally finite cover of $X$, we obtain that $g$ is a bounded continuous function on $X$.
 
 Let now $\varphi \in C_c(X) \cap D(\cE)$ with $0 \leq \varphi \leq 1$. There exists $N$ such that $\supp \varphi \subseteq G_N$. Using that $\psi_n$ and $g_n$ are supported on $\Omega_n  = G_{n+1}\setminus \overline{G_{n-1}}$, we infer
 $$\sum_{n = 1}^{N+2} \psi_n = 1 \text{ and }   \sum_{n = 1}^{N+2} g_n = g \text{ on } G_N , $$
 so that
 $$f-g = \sum_{n = 1}^{N + 2} ( f \psi_n - g_n) \text{ on } \supp \varphi.$$
 It follows from the definition of $\Ep$  that the value of $\Ep(f-g)$ only depends on the function $f-g$ on $\supp \varphi$. We obtain
 \begin{align*}
  \Ep(f - g)^{1/2}  &= \Ep\left(\sum_{n = 1}^{N + 2} ( f \psi_n - g_n)\right)^{1/2} \leq \sum_{n = 1}^{N + 2}\cE(f\psi_n - g_n)^{1/2} \leq \varepsilon^{1/2},
 \end{align*}
 where we used Theorem~\ref{theorem:properties of concatenated forms} for the first inequality. Since $\varphi \in C_c(X) \cap D(\cE)$ was arbitrary and $\cE$ is regular, it follows from Lemma~\ref{lemma:alternative formula em} that $\Em(f-g) \leq \varepsilon.$

 (b): Let now $f \in D(\Er) \cap L^\infty(m)$ and $g$ as in the proof of (a). We prove $\Ek(f-g) \leq \varepsilon$. For this estimate  we employ Lemma~\ref{lemma:alternative formula ek}. Let $\varphi \in D(\cE)$ with $0 \leq \varphi \leq 1$ arbitrary and choose a sequence $\varphi_n \in D(\cE) \cap C_c(X)$ with $0 \leq \varphi_n \leq 1$ and $\varphi_n \to \varphi$ with respect to the form norm of $\cE$.  Since $f$ and $g$ are bounded, Lemma~\ref{lemma:algebraic properties} yields
 \begin{align*}
  \Em(\varphi_n(f-g))^{1/2} &\leq \Em(f-g)^{1/2} + \|f - g\|_\infty \Em(\varphi_n)^{1/2} \\&\leq \Em(f-g)^{1/2} + \|f - g\|_\infty \cE(\varphi_n)^{1/2}.
 \end{align*}
 In particular, the sequence $(\varphi_n(f-g))$ is $\Em$-bounded and so Theorem~\ref{lemma:properties of ek} shows 
 $$\Ek(\varphi(f-g)) \leq \liminf_{n\to \infty} \Ek(\varphi_n (f-g)). $$
 Since $f-g \in D(\Em) \cap L^\infty(m)$, this computation and Lemma~\ref{lemma:alternative formula ek} show  that it suffices to prove the estimate 
 $\Ek(\psi(f-g)) \leq \varepsilon$
 for each $\psi \in D(\cE) \cap C_c(X)$ with $0\leq \psi \leq 1$. Since $\supp \psi$ is compact, it is included in $G_N$ for some $N \in  \N$. The properties $\psi_n$ and $g_n$, cf.  (a), yield 
 \begin{align*}
  \Ek(\psi (f-g))^{1/2} &= \Ek \left(\psi \sum_{n = 1}^{N + 2} ( f \psi_n - g_n) \right)^{1/2} \\&\leq \sum_{n = 1}^{N + 2} \Ek \left(\psi  ( f \psi_n - g_n) \right)^{1/2}\\
  &\leq \sum_{n = 1}^{N + 2} \cE (f \psi_n - g_n)^{1/2}\\
  &\leq \varepsilon^{1/2}.
 \end{align*}
 For the second inequality we used the monotonicity of $\Ek$ and that $\Ek(g) \leq \cE(g)$ for $g \in D(\cE)$. This proves (b).
 
 (c) + (d): Let $f \in D(\Em_a) = D(\Em)\cap L^2(m)$ and let $g$ as in (a). It remains to proof $g \in L^2(m)$ and estimates on the $L^2$-norm of $f-g$. Fatou's lemma and the properties of $\psi_n,g_n$, cf. (a), yield
 $$\|f - g\|_2 \leq \sum_{n = 1}^\infty\|f\psi_n - g_n\|_2 \leq \varepsilon.$$
 Since $f \in L^2(m)$, this implies $g \in L^2(m)$ and the claim is proven. 
 \end{proof}

\begin{remark}
 It seems to be a new observation that continuous functions are dense in the domain of the reflect form and the main part. With the same arguments it can even be strengthened. If $\mathcal{C}$ is a special standard core for $\cE$ in the sense of \cite[Section~1.1]{FOT} and $\mathcal B \subseteq C(X)$ is an algebra such  that $\mathcal{C}$ is an algebraic ideal in $\mathcal{B}$ and 
 $$\sum_{n = 1}^\infty {\psi_n} \in \mathcal{B}$$
 for all sequences  $(\psi_n)$ in $\mathcal{C}$  for which there exist neighborhoods $U_n$ of $\supp \psi_n$ that form a locally finite cover of $X$, then $\cB \cap D(\Em)$ is dense in $D(\Em)$ with respect to $\Em$. Similar statements also hold for $\Em_a$, $\Er$ and $\Era$ with the corresponding form norms.
 
 In the example of weighted manifolds, see Example~\ref{example:mainfolds}, the algebras $C_c^\infty(M)$ and $C^\infty(M)$  satisfy the above conditions. Therefore, $C^\infty(M)$ is dense in $D(\D^{(M)})$ with respect to $\D^{(M)}$ and in $D(\D^{(M)}_a)$ with respect to the form norm. The latter assertion is nothing more than the well known fact that  $C^\infty(M)$ is dense in $W^1(M)$, which is the main result  of \cite{MeSe}.
 \end{remark}

With the help of the previous theorem and Gelfand theory of commutative $C^*$-algebras we can construct a compactification   of the underlying space such that $\Er_a$ (and also $\Em_a$) can be considered to be a regular Dirichlet form on this compactification (or the compactification minus one point). Since this type of construction is essentially known, see e.g. \cite[Proof of Theorem~6.6.5]{CF} or \cite{KLSS} for Dirichlet forms on graphs, we only give a brief sketch and leave details to the reader.

Let $\cB$ a countably generated subalgebra of $D(\Er_a) \cap C_b(X)$ that is dense in $D(\Er_a)$ with respect to the form norm and in $C_c(X)$ with respect to the uniform norm. The existence of such an algebra follows from Theorem~\ref{theorem:continous functions are dense} and the separability of $(X,d)$ and   $L^2(m)$.  The (complexification of the)  uniform closure of $\cB$ equipped with the uniform norm is a commutative  $C^*$-algebra, which we denote by $\cA$. By construction $C_c(X) \subseteq \cA$ and so $C_0(X)$, the space of continuous functions that vanish at infinity, is also contained in $\cA$. Gelfand theory of commutative $C^*$-algebras  implies that there exists a compactification $\hat X$ of $X$  such that either $C(\hat X) \simeq \cA$ (this happens if $1 \in \cA$) or there exists a point $\infty \in \hat X$ such that $\cA \simeq C_0(\hat X \setminus \{\infty\})$ (this happens if $1 \not \in \cA$). Moreover, the corresponding isomorphisms are given by $C(\hat X) \to \cA$, $f \mapsto f|_X$ or $C_0(\hat X \setminus \{\infty\}) \to \cA$, $f \mapsto f|_X$, respectively.  Since $\cA$ is separable, $\hat X$ is a separable metric space and if $1 \not \in \cA$, then $\hat X \setminus \{\infty\}$ is locally compact. 

It can be proven that $1 \in \cA$ if and only if $\Ek(1) < \infty$ and $m$ is finite.  To simplify notation, in what follows we let
$$\hat X' := \begin{cases}
              \hat X &\text{if } 1 \in \cA,\\
              \hat X \setminus \{\infty\} &\text{if }1 \not \in \cA.
             \end{cases}
$$
With this convention $\cA$ is isomorphic to $C_0(\hat X')$ and the isomorphism is given by $C_0(\hat X') \to \cA,$ $f \mapsto f|_X$.

We equip $\hat X'$ with the Borel $\sigma$-algebra and extend the measure $m$ to $\hat X'$ by letting $\hat m (\hat X' \setminus X) = 0$. Then $\hat X$   is a Radon measure of full support on $\hat X'$.

The embedding map $\tau:X \to \hat X'$, $x \mapsto x$ induces a unitary operator $T: L^2(\hat X', \hat m) \to L^2(X, m)$, $f \mapsto f \circ \tau$. We define the quadratic forms $\hat \cE :L^2(\hat X',\hat m) \to [0,\infty]$, $\hat \cE(f) = \cE(Tf)$ and $\hat \Era:L^2(\hat X', \hat m) \to [0,\infty]$, $\hat \Era (f) = \Era(Tf)$. 

\begin{theorem} \label{theorem:regularity active main part}
 $\hat \Era$ is a regular Dirichlet form on $L^2(\hat X',\hat m)$.  
\end{theorem}
\begin{proof}
 Since $U$ is unitary and defined via the pointwise map $\tau$, it follows immediately that $\hat \Era$ is a Dirichlet form. Since $\cA$ is isomorphic to $C_0(\hat X')$ and the isomorphism is given by $C_0(\hat X') \to \cA, f \mapsto f|_X$, the operator $T^{-1}$ maps the algebra $\cB$ to a uniform dense subalgebra of $C_0(\hat X')$. Since $T^{-1} \cB \subseteq D(\hat \Era)$, this implies that $D(\hat \Era) \cap C_0(\hat X') \supseteq T^{-1} \cB$ is uniformly dense in $C_0(\hat X')$. The properties of $\cB$ also imply that $T^{-1} \cB$ is dense in $D(\hat \Era)$ with respect to the form norm. The regularity of $\hat \Era$ now follows from \cite[Lemma~1.4.2]{FOT}. 
\end{proof}

\begin{remark}
 For the potential theoretic notions in this remark we refer to \cite[Chapter~2]{FOT}. The form  $\hat \Era$ is regular and so there exists an associated $m$-symmetric Hunt process  on $X'$, see \cite[Theorem~7.2.1]{FOT}. Since functions in $D(\hat \cE) \cap C_c(X)$ (viewed as functions on $\hat X'$)  vanish on $\hat X' \setminus X$ and are dense in $D(\hat \cE)$ with respect to the form norm, it is also not difficult to prove the identity 
 $$D(\hat \cE) = \{f \in D(\hat \Era) \mid \tilde f = 0 \text{ q.e. on } \hat X' \setminus X\},$$
 where $\tilde f$ is a quasi-continuous version of $f$.  This situation is generic for Silverstein extensions of quasi-regular Dirichlet forms:
 
 Given a quasi-regular Dirichlet form $\cE$ on some space $X$ and a Silverstein extension $\tilde \cE$, it follows from \cite[Theorem~6.6.5]{CF} that there exists a locally compact separable metric space $\hat X$, a quasi-open subset $\tilde X$ of $\hat X$ that is quasi-homeomorphic with $X$ and a pointwise transformation (almost as above) such that $\tilde \cE$ can be considered to be a regular Dirichlet form on $\hat X$ and 
 $$D(\cE) =  \{f \in D(\tilde \cE) \mid \tilde f = 0 \text{ q.e. on } \hat X \setminus \tilde X\}.$$ 
 The new information here is that if $\cE$ is regular and $\tilde \cE = \Era$, then $\hat X$ can be chosen to be a compactification (minus one point) of $X$. For this insight the information of Theorem~\ref{theorem:continous functions are dense} is essential.
\end{remark}

\appendix

\section{Closed forms on metrizable topological vector spaces}\label{appendix:closed forms on l0}

In this section we provide a short introduction to quadratic forms on metrizable topological vector spaces. All of the results presented here are special cases of the theory developed in \cite[Chapter~1]{Schmi}, which treats quadratic forms on general topological vector spaces. Since on metrizable topological vector spaces many of the arguments in \cite{Schmi} simplify substantially, we chose to include proofs for the convenience of the reader. 

Note that all the results of this section are well known for quadratic forms on $L^2(m)$. However, their classical proofs use that $L^2(m)$ is locally convex in an essential way and in the main text we apply the theory to forms on $L^0(m)$, which is not locally convex in general.  

Let $V$ be a real vector space. We call $q:V \to [0,\infty]$ a {\em quadratic form on $V$} if it satisfies
\begin{itemize}
 \item $q(\lambda f) = |\lambda|^2 q(f)$ for all $\lambda \in \R$, $f \in V,$
 \item $q(f+g) + q(f-g) = 2q(f) + 2q(g)$ for all $f,g \in V$.
\end{itemize}
Here we use the conventions $\infty \cdot 0 = 0, \infty \cdot x = \infty$  for $x \in (0,\infty]$ and $y + \infty = \infty$ for $y \in [0,\infty]$. The  {\em domain} of a quadratic form $q$ is $D(q) = \{f \in V \mid q(f) < \infty\}$ and its {\em kernel} is $\ker q = \{f \in V \mid q(f) = 0\}$. Both sets are subspaces of $V$. For checking whether or not a functional is a quadratic form, one does not need to verify identities but certain inequalities on its domain.

\begin{lemma}\label{lemma:characterization quadratic forms}
For $q:V \to [0,\infty]$ the following assertions are equivalent.
\begin{itemize}
 \item[(i)] $q$ is a quadratic form.
 \item[(ii)] For all $\lambda \in \R$ and all $f,g \in V$ with $q(f),q(g) < \infty$ the inequalities
 $$2q(f) + 2q(g) \leq q(f + g) + q(f-g) \text{ and } q(\lambda f) \leq |\lambda|^2 q(f)$$
 hold.
\end{itemize}
\end{lemma}
By a theorem of Jordan and von Neumann, see \cite{JvN}, any quadratic form induces a bilinear form on its domain via polarization, i.e., the mapping
$$q:D(q) \times D(q) \to \R,\, (f,g) \mapsto q(f,g) := \frac{1}{4}\left(q(f+g) - q(f-g)\right)$$
is bilinear. We abuse notation and write $q$ both for the quadratic form on $V$ and the induced bilinear form on $D(q)$. In this sense, we have $q(f) = q(f,f)$ for $f \in D(q)$. Values  of the form $q(f) = q(f,f)$ are called {\em on-diagonal values of $q$} and values of the form $q(f,g)$ for different $f,g \in D(q)$ are called {\em off-diagonal values of $q$}.  It is important to note that $q$ (as a bilinear form) satisfies the Cauchy-Schwarz inequality and $q^{1/2}$ (as the square root of a quadratic form) is a seminorm on $D(q)$.

A quadratic form $\tilde{q}$ is called an {\em extension} of the quadratic form $q$ if $D(q) \subseteq D(\tilde{q})$ and $q(f)  = \tilde q(f)$ for all $f \in D(q)$. There is a natural partial order on the cone of all quadratic forms on $V$. We say that two quadratic forms $q,\tilde q$ on $V$ satisfy $q \leq \tilde q$ if $q(f) \geq \tilde q(f)$ for all $f \in V$. Large forms in terms of this relation are the ones with large domains. Indeed, we have $q \leq \tilde q$ if and only if $D(q) \subseteq D(\tilde q)$ and $q(f) \geq \tilde q(f)$ for all $f \in D(q)$.

A (real) {\em metrizable topological vector space} is a real vector space $V$ equipped with a balanced  translation invariant metric $\rho:V \times V \to [0,\infty)$, i.e., a metric such that for all $f,g,h \in V$ and $\lambda \in \R$ with $|\lambda|\leq 1$ we have
$$\rho(\lambda f,0) \leq \rho(f,0) \text{ and } \rho(f + h,g+h) = \rho(f,g).$$
The vector space operations are continuous with respect to such a metric. For a sequence $(f_n)$ in $(V,\rho)$ that converges to $f \in V$ with respect to $\rho$ we write $f_n \overset{\rho}{\to} f$ for short.
\begin{remark}
 In the text we shall mainly consider  $V = L^2(m)$ or $V = L^0(m)$ for some $\sigma$-finite measure $m$ on $X$. In the latter case a translation invariant metric is given as follows.  For an  ascending sequence of sets of finite measure  $(F_n)$ with $\cup_n F_n = X$ we let
 $$d(f,g) := \sum_{n = 1}^\infty \frac{1}{2^n m(F_n)} \int_{F_n} |f-g| \wedge 1 d m.$$
 It is a translation invariant and balanced metric and it induces the topology of local convergence in measure. Moreover, the metric space $(L^0(m),d)$ is complete.
\end{remark}
In what follows we fix a metrizable topological vector space $(V,\rho)$. The following lemma is nontrivial because balls with respect to $\rho$ need not be convex.

\begin{lemma}[Convergent Ces\`aro means] \label{lemma:convergent cesaro means}
 Let $(f_n)$ be a convergent sequence in $(V,\rho)$ with limit $f$. Then there exists a subsequence $(f_{n_k})$ such that for all of its subsequences $(f_{n_{k_l}})$ we have
 \
 $$\lim_{N \to \infty} \frac{1}{N} \sum_{l = 1}^N f_{n_{k_l}} = f.$$
\end{lemma}
\begin{proof} Since $\rho$ is translation invariant, we may assume $f=0$. We choose a subsequence $(f_{n_k})$ such that 
$$\sum_{k = 1}^\infty \rho(f_{n_k},0) < \infty.$$
For an arbitrary subsubsequence $(f_{n_{k_l}})$  the translation invariance of $\rho$  implies
$$\rho\left(\frac{1}{N} \sum_{l = 1}^N f_{n_{k_l}}, 0\right) \leq \rho\left(\frac{1}{N} \sum_{l = 1}^M f_{n_{k_l}},0\right) + \sum_{l = M + 1} ^N\rho\left(\frac{1}{N}f_{n_{k_l}},0\right).$$
From this inequality and the balancedness of $\rho$ we infer
$$\rho\left(\frac{1}{N} \sum_{l = 1}^N f_{n_{k_l}}, 0\right) \leq \rho\left(\frac{1}{N} \sum_{l = 1}^M f_{n_{k_l}},0\right) + \sum_{l = M + 1} ^\infty \rho\left(f_{n_{k_l}},0\right).$$
Choosing $M$ large enough and then using the continuity of the multiplication with scalars at $0$ finishes the proof.   
\end{proof}

We call a quadratic form $q$ on a metrizable topological vector space $(V,\rho)$ {\em closed} if it is lower semicontinuous with respect to $\rho$-convergence, i.e., if $f_n \toq f$ implies 
$$q(f) \leq \liminf_{n \to \infty} q(f_n).$$
A quadratic form is called {\em closable} if it possesses a closed extension. For determining whether or not a quadratic form possesses a closed extension the following lemma is useful.
\begin{lemma}\label{lemma:characterization closability}
 A quadratic form $q$ on $(V,\rho)$ is closable if and only if it is lower semicontinuous on its domain, i.e., if for all sequences $(f_n)$ in $D(q)$ and $f \in D(q)$ the convergence $f_n \overset{q}{\to} f$ implies
 $$q(f) \leq \liminf_{n\to \infty} q(f_n).$$
 In this case, it possesses a smallest closed extension $\bar q:V \to [0,\infty]$ that is given by
 $$\bar q (f) = \begin{cases}
                 \lim\limits_{n \to \infty} q(f_n) &\text{if } (f_n) \text{ is }q\text{-Cauchy with } f_n \toq f,\\
                 \infty &\text{if there exists no such sequence}.
                \end{cases}
$$
\end{lemma}
\begin{proof}
 The necessity of lower semicontinuity on the domain for the existence of a closed extension is clear and so it suffices to show that this condition is sufficient. Recall that a function $F:V \to (-\infty,\infty]$ is lower semicontinuous if and only if its epigraph $${\rm epi}\, F = \{(f,t) \in V \times \R \mid F(f) \leq t\}$$ is closed in the product space  $V \times \R$ (indeed this is true for functions on any metric space).
 
 Let $\overline{{\rm epi}\,q}$   the closure of ${\rm epi}\,q$ in $V \times \R$. We define $\tilde q:V \to [0,\infty]$ by 
 $$\tilde{q}(f) := \inf \{t \mid (f,t) \in \overline{{\rm epi}\,q}\},$$
 where we use  the convention $\inf \emptyset = \infty$. It is readily verified that ${\rm epi}\, \tilde q = \overline{{\rm epi}\,q}$ and so $\tilde q$ is lower semicontinuous. The lower semicontinuity of $q$ on its domain implies that $\tilde q (f) = q(f)$ for $f \in D(q)$  and that the domain of any closed extension of $q$  needs to contain the epigraph of $\tilde q$. 
 
 Let $\bar q$ as in the statement of the lemma. The lower semicontinuity of $q$ on its domain shows that $\bar q$ is well defined. Indeed, if $(f_n),(g_n)$ are $q$-Cauchy sequences in $D(q)$ with $f_n \toq f$, $g_n \toq f$, then $(f_n - g_n)$ is a $q$-Cauchy sequence with $f_n - g_n \toq 0$. Since $f_n - g_n \in D(q)$ and $q$ is lower semicontinuous on its domain, we obtain 
 $$|q(f_n)^{1/2} - q(g_n)^{1/2}| \leq q(f_n - g_n)^{1/2} \leq \liminf_{m \to \infty} q(f_n - g_n - f_m + g_m)^{1/2},$$ 
 and therefore
 $$\lim_{n\to \infty} q(f_n) = \lim_{n\to \infty} q(g_n).$$
 To finish the proof it now suffices to show $\bar q = \tilde q$, which implies that $\bar q$ is lower semicontinuous. That $\bar q$ is a quadratic form is  immediate from the definition of $\bar q$ and its minimality follows from the minimality of $\tilde q$.
 
 Let $f \in V$. We first prove $\tilde q(f) \leq \bar q(f)$. If $\bar q(f) = \infty$ there is nothing to show. Assume that there is a $q$-Cauchy sequence $(f_n)$ in $D(q)$ with $f_n \toq f$. The lower semicontinuity of $\tilde q$ and that it coincides with $q$ on $D(q)$ implies
 $$\tilde q (f) \leq \liminf_{n\to \infty} \tilde q(f_n) = \liminf_{n\to \infty} q(f_n) = \bar q (f).$$
 For proving the opposite inequality we can assume $\tilde q(f) < \infty$. Since $(f,\tilde q(f)) \in {\rm epi} \, \tilde q = \overline{{\rm epi}\,q}$,  the definition of $\overline{{\rm epi}\,q}$ yields that there exists a sequence $(f_n,t_n) \in {\rm epi} \,   q $ with $f_n \toq f$ and $t_n \to \tilde q(f)$. In particular, $q(f_n) \leq t_n$ and so the sequence $(f_n)$ is $q$-bounded. By Lemma~\ref{lemma:convergent cesaro means} and by the Banach-Saks theorem, see e.g. \cite[Theorem~A.4.1]{CF}, we can pass to a suitable subsequence and additionally assume that the sequence of Cesàro means 
 $$g_N := \frac{1}{N}\sum_{n=1}^N f_n $$
 is $q$-Cauchy and satisfies $g_N \toq f$.  Using the definition of $\bar q$ we obtain
 $$\bar q(f)^{1/2} = \lim_{N \to \infty} q(g_N)^{1/2} \leq \liminf_{N \to \infty} \frac{1}{N} \sum_{n = 1}^N q(f_n)^{1/2} = \tilde q(f)^{1/2}.  $$
 This finishes the proof.
\end{proof}

When the underlying topological vector space $(V,\rho)$ is complete, we can give another useful characterization of closed forms. To this end, we consider 
$$\rho_q :D(q) \times D(q) \to [0,\infty),\, \rho_q(f,g):= \rho(f,g) + q(f-g)^{1/2}.$$
Since $q^{1/2}$ is a seminorm, it is a translation invariant balanced metric on $D(q)$. The topology generated by $\rho_q$ is called the {\em form topology} and we call a Cauchy sequence with respect to $\rho_q$ a {\em Cauchy sequence with respect to the form topology}. The latter naming is a bit imprecise as there is no Cauchyness with respect to topologies. However, since the form topology is a vector space topology, it has a canonical uniform structure which coincides with the uniform structure induced by the metric $\rho_q$. 
\begin{lemma}\label{lemma:completeness v.s. lower semicontinuity}
 Let $(V,\rho)$ be complete. The following assertions are equivalent.
 \begin{itemize}
  \item[(i)] $q$ is a closed quadratic form on $(V,\rho)$. 
  \item[(ii)] $(D(q),\rho_q)$ is a complete metric space.
 \end{itemize}
\end{lemma} 
\begin{proof}
 (i) $\Rightarrow$ (ii): Let $(f_n)$ in $D(q)$ Cauchy with respect to $\rho_q$. Since $(V,\rho)$ is complete it has a $\rho$-limit $f \in V$. The lower semicontinuity of $q$ implies $f \in D(q)$ and
 $$q(f - f_n) \leq \liminf_{m \to \infty} q(f_m - f_n).$$
 This shows $f_n \to f$ with respect to $\rho_q$ and the completeness of is proven.
 
 (ii) $\Rightarrow$ (i): Let $(f_n)$ a sequence in $V$ and $f \in V$ with $f_n \overset{\rho}{\to} f$. We prove $q(f) \leq \liminf_n q(f_n)$. After passing to a suitable subsequence we can assume
 $$\liminf_{n \to \infty} q(f_n) = \lim_{n\to \infty}q(f_n) < \infty.$$
 The Banach-Saks theorem and Lemma~\ref{lemma:convergent cesaro means}  imply that after passing to a further subsequence we can assume that the sequence of Cesàro means 
 $$g_N := \frac{1}{N}\sum_{n=1}^N f_n$$
 is $q$-Cauchy and satisfies $g_N \overset{\rho}{\to} f$, as $N \to \infty$. The completeness of $(D(q),\rho_q)$ yields $g_N \to f$ with respect to $q$ and we obtain 
 $$q(f)^{1/2} = \lim_{N \to \infty} q(g_N)^{1/2} \leq \liminf_{N\to \infty} \sum_{n = 1}^N q(f_n)^{1/2} = \lim_{n\to \infty} q(f_n)^{1/2}.$$
 This finishes the proof.
\end{proof}

The following lemma on weakly convergent sequences is quite useful, for an $L^2$-version, see e.g. \cite[Lemma I.2.12]{MR}.

\begin{lemma} \label{lemma:existence of a weakly convergent subnet}
 Let $q$ be a closed quadratic form on $(V,\rho)$. Let $(f_n)$ be a $q$-bounded sequence in $D(q)$ and let $f \in V$  with $f_n \toq f$. Then $f \in D(q)$ and $(f_n)$ converges $q$-weakly $f$. If, additionally,
 $$\limsup_{n\to \infty} q(f_n) \leq q(f),$$
 then $f_n \to f$ with respect to $q$.
\end{lemma}
\begin{proof}
 We first show the statement on weak convergence. The lower semicontinuity of $q$ and the $q$-boundedness of $(f_n)$ imply $f \in D(q)$. Hence, by considering the sequence $(f-f_n)$ instead of $(f_n)$, we can assume $f = 0$. 
 
 The boundedness of $(f_n)$ implies that for each $g \in D(q)$ we have 
 $$- \infty < \liminf_{n\to \infty} q(f_n,g) \leq \limsup_{n\to \infty} q(f_n,g)  < \infty. $$
 Let $M \geq 0$ such that $q(f_n) \leq M$ for each $n$. Since the vector space operations are continuous with respect to $\rho$, for  $\alpha > 0$ and $g \in D(q)$  we have $g-\alpha f_n \toq g$.  The lower semicontinuity of $q$   yields
 \begin{align*}
  q(g) &\leq \liminf_{n \to \infty} q(g - \alpha f_n)\\
  &= \liminf_{n \to \infty}  \left(q(g) - 2\alpha q(f_n,g) + \alpha^2 q(f_n) \right) \\
  &\leq \liminf_{n \to \infty}  \left(q(g) - 2\alpha q(f_n,g) + \alpha^2 M \right) \\
  &=q(g) - 2\alpha \limsup_{n \to \infty} q(f_n,g) + \alpha^2 M.
 \end{align*}
Hence, for all $\alpha > 0$ we obtain $$2\limsup_{n \to \infty} q(f_n,g) \leq \alpha M,$$ which implies $$\limsup_{n \to \infty} q(f_n,g) \leq 0.$$  Since $g \in D(q)$ was arbitrary, we also have $\limsup_{n \to \infty} q(f_n,-g) \leq 0$ and conclude 
$$0 \leq  \liminf_{n \to \infty} q(f_n,g) \leq \limsup_{n \to \infty} q(f_n,g) \leq 0.$$
This shows weak convergence. The strong convergence under the additional condition $\limsup_{n\to \infty} q(f_n) \leq q(f)$ follows from the weak convergence and the identity
$$q(f-f_n) =  q(f) + q(f_n)  - 2 q(f,f_n).$$
This finishes the proof.
\end{proof}

\section{A lemma on monotone forms} \label{appendix:monotone forms} We call a quadratic form $q:L^0(m) \to [0,\infty]$ {\em nonnegative definite} if for $f,g \in D(q)$ the inequality $fg \geq 0$ implies $q(f,g) \geq 0$. We call it {\em monotone} if for $f,g \in D(q)$ the inequality $|f| \leq |g|$ implies $q(g) \leq q(f)$. The following lemma shows that monotone and nonnegative definite quadratic forms coincide when their domain is a lattice in the sense of the natural order on $L^0(m)$, i.e.,  if $f,g \in D(q)$ implies $f \wedge g \in D(q)$ and $f \vee g \in D(q)$.

\begin{lemma}\label{lemma:characterization monotonicity}
 Let $q$ be a quadratic form on $L^0(m)$ such that $D(q)$ is a lattice. The following assertions are equivalent.
 \begin{itemize}
  \item[(i)] $q$ is monotone.
  \item[(ii)] $q$ is nonnegative definite.
 
 \end{itemize}
\end{lemma}
\begin{proof}
 (i) $\Rightarrow$ (ii): For  $f,g \in D(q)$ with $fg \geq 0$ we have $|f+g| \geq |f-g|$. The monotonicity of $q$ implies
 $$q(f) + q(g) - 2q(f,g) = q(f-g) \leq q(f+g) = q(f) + q(g) + 2q(f,g)$$
 proving that it is nonnegative definite.
 
 (ii) $\Rightarrow$ (i): We first prove that $q(f) = q(|f|)$ for $f \in D(q)$. Since $D(q)$ is a lattice, $f_+ = f \vee 0$ and $f_- = (-f)\vee 0$ belong to $D(q)$. We obtain
 $$q(f) = q(f_+) + q(f_-) - 2q(f_+,f_-) \text{ and } q(|f|) = q(f_+) + q(f_-) + 2q(f_+,f_-). $$
 Since $f_+f_- = 0$, the positive definiteness of $q$ yields $q(f_+,f_-) = 0$, proving $q(f) = q(|f|)$.
 
 Now, let $f,g \in D(q)$ with $|f| \leq |g|$ be given. Using the inequalities $|f|(|g|-|f|) \geq 0$ and $|g|(|g|-|f|) \geq 0$, the positive definiteness of $q$ implies
 $$q(f) = q(|f|) \leq q(|g|,|f|) \leq q(|g|) = q(g).$$
This finishes the proof.
\end{proof}


 \bibliographystyle{plain}
 
\bibliography{literatur}

\end{document}